\documentclass[letterpaper,leqno,12pt]{amsart}
\usepackage{latexsym,amssymb,amsmath,amsthm,color}
\usepackage{amsfonts}

\newcommand{\nn}{\,\hbox{\vrule width 1pt height 8ptdepth 2pt}\,}
\newcommand\n[1]{\nn #1 \nn}

\advance\textwidth by 2.5cm %%3.6
\advance\oddsidemargin by -1.6cm \advance\evensidemargin by -1.6cm

\numberwithin{equation}{section}
\newcommand{\eps}{\varepsilon}
\newcommand{\rH}{\mathrm{ H}}

\newtheorem{theorem}{Theorem}[section]
\newtheorem{lemma}[theorem]{Lemma}
\newtheorem{proposition}[theorem]{Proposition}
\newtheorem{corollary}[theorem]{Corollary}
\newtheorem{definition}[theorem]{Definition}
\newtheorem{remark}[theorem]{Remark}

\newcommand{\sol}{\mathrm{sol}}
\newcommand\divv{{\rm div}\,}

\begin{document}
\title[]{Stationary solutions  for stochastic damped Navier-Stokes equations in $\mathbb R^d$}
\author[Z. Brze\'zniak and B. Ferrario]{Zdzis{\l}aw Brze\'zniak \\Department of Mathematics\\
                         The University of York\\
                         Heslington, York YO10 5DD, UK
       \and  Benedetta Ferrario \\Dip. di Matematica ``F. Casorati''\\
                           Universit\`a di Pavia\\
                           I-27100 Pavia, Italy}
\date{\today}
\begin{abstract}
We consider the stochastic damped Navier-Stokes equations in $\mathbb
R^d$ ($d=2,3$), assuming  as in our previous work \cite{noi} that the
covariance of the noise  is not too regular, so It\^o calculus cannot
 be applied in the space of finite energy vector fields.
We prove the existence of an invariant measure when $d=2$
and of a stationary solution when $d=3$.
\end{abstract}


\maketitle

\section{Introduction}
We consider the  stochastic damped Navier-Stokes equations,
that is the equations of motion of a viscous
incompressible fluid with two forcing terms, one is random  and the
other one is deterministic. These equations are
\begin{equation}\label{sist:ini}
\begin{cases}
\partial_t v+[-\nu \Delta v +\gamma v+(v \cdot \nabla)v +\nabla p]\ dt
= G(v)\ \partial_t w+f\ dt
\\
\divv  v=0
\end{cases}
\end{equation}
where the unknowns are the vector velocity $v=v(t,\xi)$
and  the scalar pressure $p=p(t,\xi)$ for $t\ge 0$ and  $\xi \in \mathbb R^d$.
By $\nu>0$ we denote  the kinematic viscosity  and by $\gamma> 0$  the
 sticky  viscosity, see \cite{Gal}, \cite{CR}, \cite{BF14}. When
 $\gamma=0$ they reduce to the stochastic Navier-Stokes equations.
On the right hand side $ \partial_t w$ is a space-time white
noise and $f$ is a deterministic forcing term; we consider a
multiplicative term
$G(v)$ keeping track of the fact that the noise may depend on the
velocity.  The low regularity of this term $G(v)$ is the peculiarity of our problem.

The current paper is a natural continuation of our recent  work \cite{noi} in which we dealt with  SNSEs \eqref{sist:ini} when $\gamma=0$. A common feature of both these papers is that the diffusion coefficient  $G(v)$ is 
 not regular enough to allow to use the It\^o formula
in the space  of finite energy velocity vector fields. In \cite{noi} 
 we proved the
existence of martingale solutions, for $d=2$ or $d=3$ and, for $d=2$, the pathwise
uniqueness of solutions.

Our  aim here is to investigate the existence of  invariant measures in $\mathbb R^2$
and stationary solutions  in $\mathbb R^3$ for these stochastic
damped Navier-Stokes equations \eqref{sist:ini}.
 On one side we follow the work of \cite{FG}, where the existence of
 stationary martingale solutions has been proved in two dimensional
 bounded domains when the noise term is of low regularity. 
 On the other side we follow the method used by the first
named authour with Motyl and Ondrej\'at \cite{BMO_2015} for two
dimensional  unbounded  domains  but
 with  more regular noise. The proof of the existence of an invariant
 measure given in the latter paper
 is based on the technique working with weak topologies  introduced by  Maslowski and Seidler \cite{Maslowski+Seidler_1999}.  
 This technique has been successful also
in the study of invariant measures for stochastic nonlinear beam and
wave equations considered in \cite{BOS}.

We would like to  point out that similarly to our previous paper \cite{noi} by  using $\gamma$-radonifying (instead of Hilbert-Schimdt) operators and It\^o integral in suitable  Banach spaces we were able to  weaken the   assumptions from \cite{FG}, even in the bounded domain cases,   on the covariance of the noise,  for the existence of  an invariant measure and a stationary  solution.

Let us recall that the existence of an invariant measure for
2d stochastic Navier-Stokes equations driven by an  additive noise
in  unbounded domains  satisfying the Poincar\'e condition has been
proved (as a byproduct of  the existence of an invariant compact
random set) by the first named authour and Li in
\cite{Brzezniak+Li_2006},
see also  \cite{bclllr}.  The same proof would work,
when the noise is additive, for the
damped  stochastic Navier-Stokes equations in $\mathbb{R}^2$
considered in the present article
and  for stochastic reaction diffusion
equations. Finally, we point out that recently Wang
in \cite{w_2002} considered the 2d   stochastic Navier-Stokes
equations  in unbounded Poincar\'e domains and with
a linear multiplicative noise in Stratonovich form (for which the Doss-Sussman
transformation is applicable).
So with respect to these papers, our work concerns more general noise
terms.

As far as the contents of the paper are concerned, in Section 2 we
introduce the spaces and operators appearing in the abstract
formulation of system \eqref{sist:ini}, the assumptions on the
noise term, some  compactness results and definitions.
In Section 3 we work in $\mathbb R^2$ and obtain the existence
of an invariant measure. In Section 4  we work in $\mathbb R^3$ and obtain
the existence of a stationary solution, as limit of a sequence of
stationary martingale solutions solving a  smoothed equation.

%%%%%%%%%%%%%%%%%%%%%%%%%%%%%%%%%%%%%%
\section{Mathematical framework}\label{sec-mf}
We  recall the basic notations and results. For more details we refer
to  the monograph \cite{temam} by Temam and our paper \cite{noi}.

%%%%%%%%%%%%%%%%%%%%%%%%%%%%%%%%%%%%%%%%%%%%%%%
\subsection{Spaces and operators}
For $1 \le p<\infty$ let $L^p=[L^p(\mathbb R^d)]^d$ with norm
\[
\|v\|_{L^p}=\left(\sum_{k=1}^d \|v^k\|_{L^p(\mathbb R^d)}^p\right)^{\frac1p}
\]
where $v=(v^1,\ldots,v^d)$.  For $p=\infty$, we set
$\|v\|_{L^\infty}=\sum_{k=1}^d \|v^k\|_{L^\infty(\mathbb R^d)}$. \\
Set $J^s=(I-\Delta)^{\frac s2}$.
We define the generalized Sobolev spaces of divergence free vector
distributions as
\begin{eqnarray}
H^{s,p}&=&\{u \in[{\mathcal S}^\prime(\mathbb R^d)]^d:
\|J^s u\|_{L^p}<\infty\},
\\
H^{s,p}_\sol&=&\{u \in H^{s,p} : \divv  u =0 \}
\end{eqnarray}
for $s \in \mathbb R$ and $1\le p \le \infty$.
$J^\sigma$ is an isomorphism between $H^{s,p}$ and $H^{s-\sigma,p}$
for $s \in \mathbb R$ and $1< p < \infty$.
Moreover $H^{s_2,p} \subset H^{s_1,p}$ when $s_1<s_2$.
In particular, for the Hilbert case $p=2$
we set $\rH=H^{0,2}_\sol$ and, for $s\neq 0$, $\rH^s=H^{s,2}_\sol$;
that is
\[
 \rH=\{v \in[L^2(\mathbb R^d)]^d: \divv  v =0\}
\]
with scalar product inherited from $[L^2(\mathbb R^d)]^d$. Moreover,
$H_{\mathrm{loc}}$ is the  space $H$ with the topology
generated by the family of
semi-norms $\|v\|_{H_N}=\left(\int_{|\xi|<N}|v(\xi)|^2d\xi\right)^{1/2}$,
$N\in \mathbb N$, and $L^2(0,T;H_{\mathrm{loc}})$ is the space $L^2(0,T;H)$ with the topology
generated by the family of semi-norms $\|v\|_{L^2(0,T;H_N)}$,
$N\in \mathbb N$.
By $H_{\mathrm{w}}$  we denote the space $H$  with the weak topology and by
$C([0,T];H_{\mathrm{w}})$ the space of $H$-valued  weakly continuous
functions with the topology of uniform weak convergence on $[0,T]$;
in particular $v_n \to v$ in $C([0,T];H_{\mathrm{w}})$ means
\[
\lim_{n\to \infty} \sup_{0\le t\le T}|(v_n(t)-v(t),h)_H|=0
\]
for all $h \in H$.
Notice that $v(t) \in H$ for any $t$ if $v \in C([0,T];H_{\mathrm{w}})$.

From \cite{HW} one knows that there exists a  separable Hilbert space $U$
such that $U$ is a dense subset
of $H^1$ and is compactly embedded in $H^1$.
We also have that
\[
U \subset H^1 \subset H \simeq H^\prime\subset H^{-1}\subset U^\prime
\]
with dense and continuous embeddings, but in addition $H^{-1}$
is compactly embeddded in $U^\prime$.

Now we define the operators appearing in the abstract formulation.
Let $A=-\Delta$; then
$A$ is a
linear unbounded  operator in $H^{s,p}$ as well as in
$H^{s,p}_\sol$
($s \in \mathbb R$, $1\le p<\infty$), which generates
a contractive and analytic $C_0$-semigroup $\{e^{-tA}\}_{t\ge 0}$.
Moreover, for $t>0$ the operator $e^{-tA}$ is bounded
from $H^{s,p}_\sol$ into $H^{s^\prime,p}_\sol$
with $s^\prime>s$ and there exists a constant $M$ (depending on $s^\prime-s$ and $p$)
such that
\begin{equation}\label{semigruppo}
\|e^{-tA}  \|_{\mathcal L(H^{s,p}_\sol;H^{s^\prime,p}_\sol)}\le M (1+t^{-(s^\prime-s)/2}) .
\end{equation}
We have $A: H^1 \to  H^{-1}$ as a linear bounded operator and
\[
\langle Av,v\rangle=\|\nabla v\|_{L^2}^2, \;\;\; v\in H^1,
\]
where
\[
\|\nabla v \|_{L^2}^2=\sum_{k=1}^d \|\nabla v^k\|^2_{L^2}, \;\;\; v\in H^1
\]
and, in general, $\langle\cdot, \cdot \rangle$ denotes the
$(H^{s,p})^\prime-H^{s,p}$ duality bracket.

Moreover
\begin{equation}\label{quadrati}
\|v\|_{H^1}^2=\|v\|^2_{L^2}+\|\nabla v\|_{L^2}^2 .
\end{equation}

We define the bilinear operator $B:H^1\times H^1\to H^{-1}$ via the
trilinear form
\[
\langle B(u,v),z\rangle=\int_{\mathbb R^d} (u(\xi)\cdot \nabla )v(\xi) \ \cdot z(\xi) \ d\xi.
\]
This operator $B$  is bounded and
\begin{equation}\label{scambio}
\langle B(u,v),z\rangle =-\langle B(u,z),v\rangle , \qquad
\langle B(u,v),v\rangle =0
\end{equation}
for any $u,v,z \in H^1$. Moreover,
$B$ can be extended to a bounded bilinear operator from
$H^{0,4}_\sol\times H^{0,4}_\sol$ to $H^{-1}$ with
\begin{equation}\label{bL4}
\|B(u,v)\|_{H^{-1}}\le \|u\|_{L^4}\|v\|_{L^4}
\end{equation}
and, for any $a>\frac d2+1$, $B$ can be extended to a bounded bilinear operator from
$H\times H$ to $H^{-a}$ with
\begin{equation}\label{stimaB-a}
\|B(u,v)\|_{H^{-a}}\le C \|u\|_{L^2}\|v\|_{L^2}.
\end{equation}
Once and for all we denote by $C$ a generic constant, which may vary
from line to line; we number it if
we need to identify it.

Finally, we define the noise forcing term.
Given a real separable Hilbert space $Y$
we consider a $Y$-cylindrical Wiener process $\{w(t)\}_{t \ge 0}$ defined on a
stochastic basis $(\Omega,\mathbb F, \{\mathbb F_t\}_{t\ge 0},\mathbb P)$
where the filtration  is  right continuous.

For the covariance of the noise we make the following assumptions:
there exists $g \in (0,1)$ such that
\begin{enumerate}
\item[{\bf (G1)}]
the mapping
$G:H\to \gamma(Y;H^{- g })$
          is well defined and
\[
\sup_{v \in H} \|G(v)\|_{\gamma(Y;H^{-g})}=:K_{g,2}<\infty
\]
\item[{\bf (G2)}]
 the mapping  $G:H\to \gamma(Y;H^{- g ,4}_\sol)$
            is  well defined  and
\[
\sup_{v \in H} \|G(v)\|_{\gamma(Y;H^{-g,4}_\sol)}=:K_{g,4}<\infty
\]
\item[{\bf (G3)}] if assumption {\bf (G1)} holds, then for any
$\phi \in H^{-g}$ the mapping $H\ni v\mapsto G(v)^* \phi \in Y$ is continuous
when in  $H$ we consider the Fr\'echet topology inherited from
the space $H_{\mathrm{loc}}$ or the weak topology of $H$
\item[{\bf (G4)}] if assumption {\bf (G1)} holds,
then $G$ extends to a Lipschitz continuous map
$G:H^{-g}\to \gamma(Y;H^{-g})$, i.e.
\[
\exists\ L_g>0: \; \|G(v_1)-G(v_2)\|_{\gamma(Y;H^{-g})}
    \le L_g\|v_1-v_2\|_{H^{-g}}
\]
for any $v_1, v_2 \in H^{-g}$
\end{enumerate}
We remind that the case $g=0$ in {\bf (G1)} would give a
Hilbert-Schmidt operator in $H$, which is the case considered in
\cite{BM2013,BMO_2015}.

%%%%%%%%%%%%%%%%%%%%%%%%%%%%%%%%%%%%%%%%%%%%%%%
\subsection{Abstract formulation}
Projecting equation \eqref{sist:ini} onto the space of divergence
free vector fields, we get
the  abstract form of the stochastic damped
Navier-Stokes equations \eqref{sist:ini}
\begin{equation}\label{sns}
dv(t)+[Av(t)+\gamma v(t)+B\left(v(t),v(t)\right)]\,dt
=
G(v(t))\,dw(t)+f(t)\ dt
\end{equation}
with initial velocity $v(0):\Omega\to H$ having the  law $\mu$.
Here $\gamma>0$ is fixed  and for simplicity we have put $\nu=1$.
The case $\gamma=0$ was considered in our previous paper \cite{noi}. Now we add
the damping term in order to investigate the  existence of a stationary solution; this is
not necessary in Poincar\'e (in particular  bounded) domains.

We define a martingale solution on a finite time interval, assuming
{\bf (G1)} and $f \in L^1(0,T;H^{-1})$, $\mu(H)=1$.
\begin{definition}
[martingale solution]
We say that there exists a martingale solution of equation \eqref{sns}
on the time interval $[0,T]$ with initial velocity of law $\mu$ if there exist
\begin{enumerate}
\item[$\bullet$]
a stochastic basis $(\hat\Omega,\hat{\mathbb F},
\{\hat{\mathbb F}_t\}_{t\in [0,T]},\hat{\mathbb P})$
\item[$\bullet$]
a $Y$-cylindrical Wiener process $\hat w$
\item[$\bullet$]
a progressively measurable process $\hat v:[0,T]\times \hat \Omega\to H$ with
$\hat{\mathbb P}$-a.e. path
\[
\hat v \in C([0,T];H_{\mathrm{w}})\cap L^2(0,T;L^4)
\]
and $\hat v(0)$ has law $\mu$; moreover for any $t \in [0,T], \psi \in H^2$, $\hat{\mathbb P}$-a.s.,
\begin{eqnarray}\nonumber \label{sol-path}
 ( \hat v(t),\psi)_H
 &+&\int_{0}^t \langle A\hat v(s) ,\psi\rangle ds
 +\gamma \int_{0}^t ( \hat v(s) ,\psi)_H ds
\\\nonumber
 &+&\int_{0}^t \langle B(\hat v(s),\hat v(s)),\psi \rangle ds
\\
&=&
 ( \hat v(0),\psi)_H
  +\int_{0}^t \langle f(s) ,\psi\rangle ds
  +\langle \int_{0}^t G(\hat v(s))\,d\hat w(s),\psi\rangle.
\end{eqnarray}
\end{enumerate}
\end{definition}
\noindent
All the terms in \eqref{sol-path} make sense; in particular
the trilinear term is well defined thanks to \eqref{bL4} which
 provides the following estimate
\[
\Big|\int_{0}^t \langle B(\hat v(s),\hat v(s)),\psi \rangle ds\Big|
\le
\|\psi\|_{H^1}\int_0^t\|\hat v(s)\|_{L^4}^2ds.
\]
However, in the following sections we shall prove the existence of a martingale
solution with a time integrability higher than
$\hat v \in L^2(0,T;L^4)$.
As far as the stochastic integral is concerned, by assuming {\bf (G1)}
we obtain that the random variable
$\int_{0}^t G(\hat v(s))\,d\hat w(s)$ belongs to
$H^{-g}$ ($\hat{\mathbb P}$-a.s.), see
Proposition 5.2 in \cite{noi}.
\\
Initial deterministic velocity   $x \in H$ corresponds to $\mu=\delta_x$.

As far as stationary martingale solutions are concerned, we give the definition
involving the time interval $\mathbb R_+=[0,\infty)$.
Again, we assume {\bf (G1)} and $f \in L^1_{\mathrm{loc}}(\mathbb R_+;H^{-1})$, $\mu(H)=1$.
\begin{definition}
[stationary martingale solution]
We say that there exists a stationary martingale solution of  equation  \eqref{sns}
on the time interval $\mathbb R_+$
with initial velocity of law $\mu$ if there exist
\begin{enumerate}
\item[$\bullet$]
a stochastic basis $(\hat\Omega,\hat{\mathbb F},\{\hat{\mathbb F}_t\}_{t\ge 0},\hat{\mathbb P})$
\item[$\bullet$]
a $Y$-cylindrical Wiener process $\hat w$
\item[$\bullet$]
a stationary and progressively measurable process $\hat v:\mathbb R_+\times \hat \Omega\to H$ with
$\hat{\mathbb P}$-a.e. path
\[
\hat v \in C(\mathbb R_+;H_{\mathrm{w}})\cap L^2_{\mathrm{loc}}(\mathbb R_+;L^4)
\]
and $\hat v(0)$ has law $\mu$; moreover, for any $t >0$ and
$\psi \in H^2$ equation \eqref{sol-path} holds $\hat{\mathbb P}$-a.s.
\end{enumerate}
\end{definition}

The first result is about the existence of  a martingale
solution. Moreover for $d=2$ there is pathwise uniqueness.
Now, our aim is to prove the existence of a martingale stationary solution  for
$d=2,3$; actually, for $d=2$ we prove something more, that is the existence of
an  invariant measure.

%%%%%%%%%%%%%%%%%%%%%%%%%%%%%%%%%%%%%%%%%%%%%%%
\subsection{Compactness results}
In this section we fix  $p \in (1,\infty)$ and $\beta,\delta \in
(0,\infty)$. For any $T>0$ let us define
the space
\begin{equation}\label{eqn-Z_T}
Z_T=C([0,T];H_{\mathrm{w}})\cap L^2(0,T;H_{\mathrm{loc}})\cap
L^{p}_{\mathrm{w}}(0,T;L^4)\cap C([0,T];U^\prime)
\end{equation}
which is a locally convex topological space with the topology
${\mathcal T}_T$ given by the supremum of the corresponding topologies. Let us also define a function
$\n{\cdot}:Z_T \to [0,\infty]$ by
\begin{equation}\label{eqn-n_T}
\n{ v}_T=\|v\|_{L^\infty(0,T;H)}+\|v\|_{L^2(0,T;H^\delta)}+
 \|v\|_{L^p(0,T;L^4)}+\|v\|_{C^\beta([0,T];H^{-1})},
\end{equation}
if the RHS above is finite, and, otherwise,  $\n{ v}_T=\infty$.
Notice that the function $\n {\cdot}_T$ is  not the  natural norm of $Z_T$. But it is used  in the following compactness result.
\begin{lemma}\label{lemmaZT}
For any $a>0$ the set
\[
K_T(a)=\{v \in Z_T: \n{ v}_T \le a\}
\]
is a metrizable compact subset of $Z_T$.
\end{lemma}
\begin{proof}
Since $K_T(a)$ is a bounded subset of $Z_T$, it is metrizable for each of the
four topologies involved in the definition of ${\mathcal T}_T$.
The compactness result comes from   Lemma 5.4 in \cite{noi}.
\end{proof}

Let us point out that we considered $L^2(0,T;H_{\mathrm{loc}})$  to be  the space
$L^2(0,T;H)$  with the topology
generated by the semi-norms $\|v\|_{L^2(0,T;H_N)}$,
$N\in \mathbb N$, since the space $K_T(a)$ involves a boundedness in
the $L^2(0,T;H^\delta)$-norm; we know that
any bounded sequence in
$\|v\|_{L^2(0,T;H^\delta)}$ has a subsequence
weakly converging and its limit belongs to $L^2(0,T;H^\delta)$
which is contained in $L^2(0,T;H)$.

Working in the unbounded time interval $\mathbb R_+$, we have a similar compactness result.
We consider the following locally convex topological spaces:
\begin{itemize}
\item
$C(\mathbb R_+;H_{\mathrm{w}})$
 with the topology generated by the family of semi-norms\\
 $\|v\|_{N,h}=\displaystyle\sup_{0\le t\le N}|\langle v(t), h\rangle |$, $N \in
 \mathbb N, h \in H$;
\item
$L^2_{\mathrm{loc}}(\mathbb R_+;H_{\mathrm{loc}})$ is the space
$L^2_{\mathrm{loc}}(\mathbb R_+;H)$ with metric \\
$d(u,v)=\displaystyle\sum_{N=1}^\infty
     \frac 1{2^N}\frac{\|u-v\|_{L^2(0,N;H_N)}}{1+\|u-v\|_{L^2(0,N;H_N)}}$;
\item
$L^p_{\mathrm{loc},\mathrm{w}}(\mathbb R_+;L^4)$ is the space $L^p_{\mathrm{loc}}(\mathbb R_+;L^4)$
 with the topology generated by the family of semi-norms, with $\frac 1p+\frac 1{p^\prime}=1$,
 \[\|v\|_{N,h}=\displaystyle \vert \int_0^N \int_{\mathbb R^3} v(t,\xi)h(t,\xi) dt d\xi\vert , \; N\in
 \mathbb N, h \in L^{p^\prime}(0,N;L^{\frac 43});\]
\item
$C(\mathbb R_+;U^\prime)$  with metric
$d(u,v)=\displaystyle\sum_{N=1}^\infty
     \frac
     1{2^N}\frac{\|u-v\|_{C([0,N];U^\prime)}}{1+\|u-v\|_{C([0,N];U^\prime)}}$.
\end{itemize}
We define the  space
\[
Z=C(\mathbb R_+;H_{\mathrm{w}})\cap L^2_{\mathrm{loc}}(\mathbb R_+;H_{\mathrm{loc}})
\cap
L^p_{\mathrm{loc},\mathrm{w}}(\mathbb R_+;L^4)
\cap C(\mathbb R_+;U^\prime).
\]
It is a locally convex topological space with the topology $\mathcal T$
given by  the supremum of the corresponding topologies.

We have this compactness result.
%LEMMA
\begin{lemma}\label{lemma-Zcomp}
For any sequence $\alpha=\{\alpha_N\}_{N \in \mathbb N}$
of positive numbers,  the set
\[
K(\alpha)=\{v \in Z:\n{ v}_N \le \alpha_N \text{
  for any } N \in \mathbb N\}
\]
is a metrizable compact subset of $Z$.
\end{lemma}
\begin{proof}
It is enough to know
that each $K_N(\alpha_N)$ is a metrizable compact subset of $Z_N$, see, e.g., a similar case in Corollary B.2 of \cite{On2010}. But this is Lemma
\ref{lemmaZT}.
\end{proof}

%%%%%%%%%%%%%%%%%%%%%%%%%%%%%%%%%%%%%%%%%%%%%%%
\subsection{Invariant measures}
Here we suppose that given the initial velocity $x\in H$ there exists a unique solution to
equation \eqref{sns}. We denote  by  $v(t;x)$  this solution at time $t>0$ and
assume that the deterministic forcing term $f$ is independent of time.
Therefore we
define the  family $\{P_t\}_{t\ge 0}$,
\begin{equation}\label{eqn-Markov}
(P_t\phi)(x)= {\mathbb E}[\phi( v(t;x))]
\end{equation}
for any $\phi\in B_b(H)$, i.e. $\phi:H\to \mathbb R$ is a  bounded and Borel measurable  function.

It is obvious that $P_t\phi$ is bounded  for every $\phi\in
B_b(H)$. Moreover,  it is also measurable when the weak existence and
uniqueness in law hold for equation \eqref{sns}, see Corollary 23 in
\cite{On2005} (which  generalizes to the infinite dimensional
setting  the finite dimensional result of Stroock and  Varadhan  \cite{Stroock+Varadhan_2006}).
Let us point out that if this unique solution to equation \eqref{sns}
has a.e. path in $C([0,T];H)$, then it is also a Markov process, see
Theorem 27 in \cite{On2005}.

First  we introduce four classes of  functions $\phi:H\to \mathbb R$
which are continuous with respect to different topologies:

\begin{itemize}
\item[(i)] A function $\phi \in C(H)$ iff  $\phi$ is continuous w.r.t. to strong
topology on $H$;

\item[(ii)] a function $\phi \in C(H_{\mathrm{w}})$ iff  $\phi$ is continuous w.r.t. to weak
topology  on $H$;

\item[(iii)] a function $\phi \in SC(H_{\mathrm{w}})$ iff  $\phi$ is sequentially continuous w.r.t. to weak
topology  on $H$

\item[(iv)] a function $\phi \in C(H_{\mathrm{bw}})$  iff $\phi$ is continuous
w.r.t. to bounded weak
topology  on $H$, which is  the finest topology on $H$,   whose family of closed sets  agrees with the family of closed sets from the weak topology on  closed balls in $H$.
\end{itemize}
It is known,  see e.g. \cite{ms2001}, that
\begin{equation}\label{eqn-C spaces}
SC(H_{\mathrm{w}})=C(H_{\mathrm{bw}}).
\end{equation}

Note also that   the following inclusions hold
\[
  C(H_{\mathrm{w}}) \subset C(H_{\mathrm{bw}}) \subset    C(H) .
\]

When we add the subscript $b$ we mean that the function is also bounded, for instance
$C_b(H_{\mathrm{bw}})=\{ \phi \in C(H_{\mathrm{bw}}): \phi \mbox{ is bounded}\}$.

Thus, we get the following inclusions
\begin{equation}\label{eqn-C spaces2}
 C_b(H_{\mathrm{w}}) \subset C_b(H_{\mathrm{bw}}) \subset C_b(H) .
\end{equation}
Let us also notice that because $H$ is a separable space,
the weak Borel and the (strong) Borel $\sigma$
fields on $H$ are equal, i.e.
$\mathcal{B}(H)=\mathcal{B}(H_{\mathrm{w}})$, see  Theorem 7.19 in
 \cite{Zizler_2003}  and  \cite{Edgar_1978} for more general claims.
On the other hand, since the strong topology is finer than the bw-topology which in turn is finer than the weak topology, we infer  that $\mathcal{B}(H)\supseteq \mathcal{B}(H_{\mathrm{bw}})\supseteq \mathcal{B}(H_{\mathrm{w}})$. Thus we proved that
\begin{equation}\label{eqn-Borel sets}
\mathcal{B}(H)= \mathcal{B}(H_{\mathrm{bw}})= \mathcal{B}(H_{\mathrm{w}}).
\end{equation}
We infer that
the sets of  $\mathcal{B}(H_{\mathrm{w}})$, $\mathcal{B}(H_{\mathrm{bw}})$ and $\mathcal{B}(H)$-measurable functions are equal. Hence
\begin{equation}\label{eqn-B spaces}
B_b(H)=B(H_{\mathrm{bw}})=B_b(H_{\mathrm{w}}).
\end{equation}

Let us now recall the following fundamental definition,  see \cite{Maslowski+Seidler_1999}.
We say that a family $\{P_t\}_{t\ge 0}$ is sequentially weakly Feller iff
\begin{equation}\label{def-swFelller}
P_t:SC_b(H_{\mathrm{w}}) \to SC_b(H_{\mathrm{w}}), \;\;\; t\ge 0.
\end{equation}

Thanks to the results of Ondrej\'at \cite{On2005} quoted before, we deduce  that
 the family $\{P_t\}_{t\ge 0}$ is also a Markov semigroup.

Given a sequentially weakly Feller  Markov semigroup on a separable Hilbert space $H$,
 we can define an invariant measure $\mu$ for equation \eqref{sns} as a
Borel probability measure  on $H$ such that for any time $t\ge 0$
\begin{equation}\label{def-invmeas}
\int_H P_t \phi\ d\mu=\int_H \phi\ d\mu, \;\; \forall  \phi\in C_b(H_{\mathrm{w}}).
\end{equation}
This definition is meaningful as
the LHS  of \eqref{def-invmeas} makes sense. Indeed, if
$\phi\in C_b(H_{\mathrm{w}})\subset  SC_b(H_{\mathrm{w}})$, 
then    by \eqref{def-swFelller} and \eqref{eqn-C spaces}
$P_t\phi\in  SC_b(H_{\mathrm{w}})=C_b(H_{\mathrm{bw}})\subset B_b
(H_{\mathrm{bw}})=B_b(H)$; therefore  the integral in the LHS is well defined.

Moreover, it is well known that
the set of continuous functions is a determining set for a Borel  measure.
In particular,
the  set  $C_b(H_{\mathrm{w}})$
is  a  determining set for the measure $\mu$, i.e. if $\mu_i$, $i=1,2$ are
two Borel probability measures on $\mathcal{B}(H_{\mathrm{w}})$ such that
$\int_H \phi\ d\mu_1=\int_H \phi\ d\mu_2$, for every $\phi\in C_b(H_{\mathrm{w}})$, then
$\mu_1=\mu_2$.
Therefore, relationship \eqref{def-invmeas} corresponds to the invariance (in
time) of the law of the random variable $v(t;x)$.

\begin{remark}
Often the definition of invariant measure is given  when the Markov
semigroup is Feller, i.e.  $P_t:C_b(H) \to C_b(H)$ for any $t\ge 0$, see, e.g. \cite{dpz}. Indeed, the set $C_b(H)$
is still a determining set for the measure $\mu$ and  the LHS  of \eqref{def-invmeas} makes sense since
$P_t\phi\in C_b(H)\subset B_b(H)$.
\end{remark}

%%%%%%%%%%%%%%%%%%%%%%%%%
\section{Invariant measures for $d=2$}
We consider equation \eqref{sns} in $\mathbb R^2$.
In \cite{noi} we proved the existence and uniqueness of solutions for
this equation when $\gamma=0$. Dealing now with the case
$\gamma>0$, we can prove in the same way the same result. Indeed,
as in \cite{noi}
we split the analysis of equation \eqref{sns} for $v$ in two
subproblems involving the processes  $z$ and $u$ (with $v=z+u$), where
\begin{equation}\label{OU}
dz(t)+Az(t)\ dt+\gamma z(t)\ dt=G(v(t))\,dw(t),\;t\in (0,T];
\qquad z(0)=0
\end{equation}
and
\begin{equation}\label{eq:u}
\frac {d u}{dt}(t) +Au(t) +\gamma u(t)+B(v(t),v(t))=f(t),\;t\in (0,T];
\qquad u(0)=x
\end{equation}
The solution of equation \eqref{OU} is
\[
z(t)=\int_0^t  e^{-\gamma(t-s)}e^{-(t-s)A}G(v(s))\,dw(s)
\]
and the basic energy equality for equation \eqref{eq:u} is
\begin{multline}
\frac 12 \frac {d}{dt}\|u(t)\|_{H}^2 +\|\nabla u(t)\|_{H}^2+\gamma \|u(t)\|^2_H
\\
= -\langle B(z(t)+u(t),z(t)+u(t)), u(t)\rangle + \langle f(t), u(t)\rangle .
\end{multline}
This shows that we can work on  equations \eqref{OU} and \eqref{eq:u}
as we did in \cite{noi} when $\gamma=0$.
In this way,
we can prove first the existence
of a martingale solution and then  the pathwise uniqueness, see Theorem 4.2 in \cite{noi}.
Hence by invoking the Yamada-Watanabe Theorem, see
\cite{IW},  we deduce the existence of a strong solution too (in the probabilistic sense).
\begin{theorem}\label{th:esist2}
Let  $d=2$. If  $x\in H$,  $f \in L^p_{\mathrm{loc}}([0,\infty);H^{-1})$ for some $p>2$ and
 assumptions {\bf (G1)}-{\bf (G2)}-{\bf (G3)} are satisfied,
then   there exists a martingale solution
$\left((\tilde\Omega,\tilde{\mathbb F}, \{\tilde{\mathbb F}_t\}_{t \in [0,\infty)},\tilde{\mathbb P}), \tilde w,\tilde v\right)$ of
equation \eqref{sns} on the time interval $[0,\infty)$ with deterministic initial velocity $x$.
Moreover,  the solution $\tilde v$ satisfies the following
\[
\tilde v \in C([0,\infty);H) \cap L^{\frac4{1-g}}_{\mathrm{loc}}([0,\infty);H^{\frac{1-g}2})
\cap L^{4}_{\mathrm{loc}}([0,\infty);L^4)
 \qquad \tilde{\mathbb P}-a.s.
\]
Finally, assuming also {\bf (G4)} there is pathwise
uniqueness for equation \eqref{sns}.
\end{theorem}
\begin{remark}\label{rem-th:esist2}
The above result has been proved on a fixed time interval. But it is not difficult to modify the proof in the spirit of our proofs of
Theorems \ref{th:mis-inv-3d-m} and \ref{th:mis-inv-3d} so that the result holds true also on the whole time interval. We have formulated our result in such a form. \\
The same comment applies to Theorem \ref{teo-snsm}.
\end{remark}
For  deterministic initial velocity $x\in H$,
we denote by $ v(t;x)$ the solution at time $t>0$.

Let us recall a result of Maslowski-Seidler \cite{Maslowski+Seidler_1999} about  the 
existence of an invariant measure. This is a modification of
the Krylov-Bogoliubov technique, see \cite{KB} and \cite{DPZ2}, the latter being successful in bounded  domains.
% THEOREM %
\begin{theorem}\label{MSprop}
Assume that\\
i) the semigroup $\{P_t\}_{t\ge 0}$ is sequentially weakly Feller in $H$;\\
ii) for any $\varepsilon >0$ there exists $R>0$ such that
\[
\sup_{T\ge 1} \frac 1T \int_0^T \mathbb P\left(\|v(t;0)\|_H>R\right)dt<\varepsilon .
\]
Then there exists  at least one invariant measure for equation \eqref{sns}.
\end{theorem}

Below we will verify  the two assumptions i) and ii)   in order to prove
% THEOREM %
\begin{theorem}\label{th:mis-inv-2d}
Let $d=2$.
If $f \in H^{-1}$ and  assumptions {\bf (G1)}-{\bf (G2)}-{\bf (G3)}-{\bf (G4)} are satisfied,
then there exists at least one invariant measure  $\mu$ for
equation \eqref{sns} and $\mu(H)=1$.
\end{theorem}

%%%%%%%%%%%%%%%%%%%%%%%%%%%%%%%%%%%%%%%%%
\subsection{Sequentially weakly Feller}\label{sect-SwF}
We need to verify, under the assumptions of Theorem
\ref{th:mis-inv-2d}, that
the Markov semigroup $\{P_t\}_{t\ge 0}$
is sequentially weakly Feller in $H$. This means that  for any $t>0$ and any bounded and sequentially weakly continuous function $\phi:H\to \mathbb R$,
\[
\text{ if } x_k \rightharpoonup x \text{ in } H, \text{ then }
P_t\phi(x_k)\to P_t\phi(x).
\]

Let us fix $0<t<T<\infty$.
We are given a sequence $\{x_k\}_{k \in \mathbb N}$ weakly convergent
in $H$ to $x$. For each $k \in\mathbb N$, let
$v_k$ be the strong solution of \eqref{sns} on the time interval $[0,T]$ with initial
velocity $x_k$, given by Theorem \ref{th:esist2}. Set $v_k=z_k+u_k$, with
\begin{equation}
\displaystyle
\left\{\begin{aligned}
&dz_k(t)+Az_k(t)\ dt+\gamma z_k(t)\ dt=G(v_k(t))\,dw(t),t\in [0,T],\\
&\\
& z_k(0)=0
\end{aligned}
\right.
\end{equation}
 and
\begin{equation}
\displaystyle
\left\{\begin{aligned}
\frac {d u_k}{dt}(t) &+Au_k(t) +\gamma u_k(t)+B(v_k(t),v_k(t))=f,\;t\in [0,T];\\
u_k(0)&=x_k.
\end{aligned}
\right.
\end{equation}

First, we look for  bounds in probability for $u_k$ and $z_k$, uniform in $k$
in order to get tightness and then convergence of the sequence
$\{v_k\}_k$. This will lead to prove that $P_t\phi(x_k)\equiv{\mathbb E}[\phi( v_k(t;x_k))]
\to {\mathbb E}[\phi( v(t;x))]\equiv P_t\phi(x)$
for any  bounded and sequentially weakly continuous function $\phi:H\to \mathbb R$.

We proceed as in \cite{noi};
actually, the role of $\gamma>0$ is negligible in this section.

As far as the processes $z_k$ are concerned,
we appeal to  Lemma 3.2 and Lemma 3.3 in \cite{noi}; indeed
\[
z_k(t)=\int_0^t  e^{-\gamma(t-s)}e^{-(t-s)A}G(v_k(s))\,dw(s)
\]
so
each  process $z_k$ depends on $k$
through the process $v_k$ only (now the operator $G$ is fixed)
and we can proceed as in \cite{noi};
the operator $A+\gamma$ is no worse than the operator $A$.
This means that, given
$\beta, \delta\ge 0$ such that
\begin{equation}\label{cond-zc}
\beta+\frac \delta 2<\frac {1-g}2,
\end{equation}
we have
\[
\sup_{k} \mathbb E\|z_k\|_{C^\beta([0,T];H^\delta)}<\infty
\qquad\text{ and } \qquad \sup_{k} \mathbb E\|z_k\|_{L^{4}(0,T;L^4)}<\infty.
\]
This implies that
 for any $\varepsilon>0$ there exist positive constants
$R_i=R_i(\varepsilon)$ , $i = 1, 2, 3,4$, such that
\begin{align}
&\displaystyle\sup_k \mathbb{P}(\|z_k\|_{C^\beta([0,T];H^{-1})}>R_1) \le \varepsilon,\\
&\displaystyle\sup_k \mathbb{P}(\|z_k\|_{L^2(0,T;H^\delta)}>R_2) \le \varepsilon,\\
&\displaystyle\sup_k \mathbb{P}(\|z_k\|_{L^\infty(0,T;H)}>R_3) \le \varepsilon,
\\
\label{z4}
&\sup_k \mathbb{P}(\|z_k\|_{L^{4}(0,T;L^4)}>R_4) \le \varepsilon.
\end{align}

We look for analogous estimates for $u_k$.
Taking the $H$-product of equation \eqref{sns} with $u_k$, we get
\begin{multline}\label{energyest}
\frac 12 \frac {d}{dt}\|u_k(t)\|_{H}^2 +\|\nabla u_k(t)\|_{H}^2+\gamma \|u_k(t)\|^2_H
\\
= -\langle B(z_k(t)+u_k(t),z_k(t)+u_k(t)), u_k(t)\rangle + \langle f, u_k(t)\rangle .
\end{multline}
Estimating the RHS as in the proof of  Proposition 3.4 in
\cite{noi} and neglecting the positive term $\gamma \|u_k(t)\|^2_H$ in the LHS, we obtain
\begin{equation}\label{vH}
\begin{split}
\sup_{0\le t\le T} \|u_k(t)\|_{H}^2
&\le
\|x_k\|_H^2 e^{\int_0^T \phi_k(r) dr}+
 \int_0^T e^{\int_s^T \phi_k(r) dr}\psi_k(s)ds
 \\
&\le e^{\int_0^T \phi_k(r) dr}\left[ \|x_k\|_H^2 +
 \int_0^T \psi_k(s)ds\right]
\end{split}
\end{equation}
where
$\phi_k(t)=1+C_1\|z_k(t)\|_{L^4}^4 \mbox{ and }
\psi_k(t)=C_2 (\|z_k(t)\|_{L^4}^4+ \|f\|_{H^{-1}}^2)$
for suitable constants $C_1$ and $C_2$ independent of $k$,
and
\begin{equation}\label{vV}
\begin{split}
\int_0^T \|\nabla u_k(t)\|_{L^2}^2 dt
&\le
\|x_k\|_H^2
       +\int_0^T \Big( \phi_k(t) \|u_k(t)\|_{H}^2+\psi_k(t)\Big)dt
\\
& \le
\|x_k\|_H^2
       +\|u_k\|_{L^\infty(0,T;H)}^2      \int_0^T \phi_k(t) dt
       +\int_0^T \psi_k(t) dt
       .
\end{split}
\end{equation}
Since the sequence $\{x_k\}_k$ is weakly convergent in $H$, we have
$\sup_k \|x_k\|_H<\infty$; then, from \eqref{z4}, \eqref{vH}, \eqref{vV}
we
get uniform bound for $\|u_k\|_{L^\infty(0,T;H)}$ and
$\|u_k\|_{L^2(0,T;H^1)}$. Proceeding as in
the proof of Proposition 3.4 in \cite{noi}, we  also get
uniform bounds for $\|u_k\|_{C^{\frac 12}([0,T];H^{-1})}$ and $\|u_k\|_{L^4(0,T;L^4)}$.

Now we sum up our results for $z_k$ and $u_k$.  Let us  choose
$\beta\in (0,\frac 12]$ and $\delta\in (0,1]$ fulfilling
    \eqref{cond-zc} so we  also have
$C^{\frac 12}([0,T];H^{-1})\subseteq C^{\beta}([0,T];H^{-1})$ and
$L^2(0,T;H^1)\subseteq L^2(0,T;H^\delta)$. Hence
for $v_k=z_k+u_k$
we have the following result:
\\
for any $\varepsilon>0$ there exist positive constants
$R_i=R_i(\varepsilon)$, $i = 5,\ldots, 8$, such that
\[
\sup_k \mathbb P(\|v_k\|_{L^\infty(0,T;H)}>R_5)\le \varepsilon ,
\]
\[
\sup_k \mathbb P(\|v_k\|_{L^2(0,T;H^{\delta})}>R_6)\le \varepsilon ,
\]
\[
\sup_k \mathbb P(\|v_k\|_{L^4(0,T;L^4)}>R_7)\le \varepsilon ,
\]
\[
\sup_k \mathbb P(\|v_k\|_{C^{\beta}([0,T];H^{-1})}>R_8)\le \varepsilon .
\]
Hence, for any $\varepsilon>0$ there exist
$R=R(\varepsilon)>0$ such that
\begin{eqnarray*}
\sup_k \mathbb P\left(\|v_k\|_{L^\infty(0,T;H)} \right.&+&\|v_k\|_{L^2(0,T;H^{\delta})}
\\
  &+&\left. \|v_k\|_{L^4(0,T;L^4)}+\|v_k\|_{C^{\beta}([0,T];H^{-1})}>R\right)\le \varepsilon.
\end{eqnarray*}
Bearing in mind Lemma \ref{lemmaZT}, we obtain that the
sequence of laws of the processes $v_k$ is tight in $Z_T$, see also Lemma 5.5 in \cite{noi}. Now we appeal to the
 Jakubowski's \cite{Ja1997} generalization  of the
Skorohod Theorem to  nonmetric spaces.
We can do so  since
there exists a countable family $\{f_i:Z_T\to \mathbb R\}$
of ${\mathcal T}_T$-continuous functions, which separate points of $Z_T$, see the proof of Corollary 3.12 in \cite{BM2013}.
Therefore,  there exist a subsequence $\{v_{k_j}\}_{j=1}^\infty$,
a stochastic basis
$(\tilde \Omega, \tilde{\mathbb F}, \{\tilde{\mathbb F}\}_{0\le t \le T}, \tilde{\mathbb P})$,
$Z_T$-valued
Borel measurable variables $\tilde v$ and $\{\tilde v_j\}_{j=1}^\infty$ such that
for any $j$ the laws of $v_{k_j}$ and $\tilde v_j$
are the same and $\tilde v_j$ converges to $\tilde v$ ($\tilde {\mathbb P}$-a.s.)
with the topology ${\mathcal T}_T$. Moreover, one proves as in \cite{noi}
that  $\tilde v$ coincides with the solution of
\eqref{sns} with initial velocity $x$.

In particular, for fixed $t$, $\tilde v_j(t;x_{k_j})$ weakly
converges in $H$ to   $\tilde v(t;x)$, $\tilde {\mathbb P}$-a.s.,
according to the fact that
$\tilde v_j \to \tilde v$ in $C([0,T];H_{\mathrm w})$.
Hence, given any   sequentially weakly
continuous function $\phi:H\to\mathbb R$, we have that
$\phi(\tilde v_{j}(t;x_{k_j})) \to \phi(\tilde v(t;x))$ $\tilde {\mathbb P}$-a.s.
and therefore, when $\phi$ is also bounded,  by invoking the Lebesque Dominated Convergence (LDC) theorem we
deduce that  $\tilde{\mathbb E}[\phi(\tilde v_{j}(t;x_{k_j}))]
 \to \tilde{\mathbb E}[\phi(\tilde v(t;x))]$.
Using that $\tilde v$ has the same law as $v$ and   $\tilde v_j$ has the same
law as $v_{k_j}$, we infer that $\mathbb E[\phi(v_{k_j}(t;x_{k_j}))]
\to \mathbb E[\phi(v(t;x))]$.
Moreover, by the uniqueness the whole sequence converges.
This proves the sequentially weakly
Feller property in $H$.

%%%%%%%%%%%%%%%%%%%%%%%%%%%%%%%%%%%%%%%%%
\subsection{Boundedness in probability}
We need to verify, under the assumptions of Theorem
\ref{th:mis-inv-2d}, that for any  $\varepsilon >0$ there exists $R=R(\eps)>0$ satisfying
\begin{equation}\label{th:bdinprob}
\sup_{T\ge 1} \frac 1T \int_0^T \mathbb P\left(\|v(t;0)\|_H>R\right)dt<\varepsilon .
\end{equation}

We define the modified Ornstein-Uhlenbeck equation
\begin{equation}\label{OUdamped}
dz(t)+Az(t)\ dt+(\gamma+\alpha) z(t)\ dt=G(v(t))\,dw(t)
\end{equation}
with an additional damping $\alpha>0$, to be chosen later on.
Here $v$ is the strong  solution to equation  \eqref{sns}
with zero initial velocity,  given in Theorem \ref{th:esist2}.
We consider the stochastic convolution integral
\begin{equation}\label{zalpha}
\zeta^\alpha(t)=\int_0^t e^{-(\gamma+\alpha)(t-s)}e^{-(t-s)A}G(v(s))\,dw(s)
\end{equation}
solving \eqref{OUdamped} with vanishing initial data and find good estimates on it.
This result is independent of the space dimension $d$.
% LEMMA %
\begin{lemma}\label{lemma-z2}
Assume conditions {\bf (G1)}-{\bf (G2)} and let $v$ be a continuous
$H$-valued process.
Then for any $\alpha\ge 0$ and $q\ge 2$
there exist positive constants $C_{\alpha,2}, C_{\alpha,q,4}$ (depending also on
$g$ and $\gamma$) such that for
the
process $\zeta^\alpha$ given by \eqref{zalpha} we have
\begin{equation}\label{stima-zeta-alpha}
 \mathbb E \|\zeta^\alpha(t)\|_H^2 \le  C_{\alpha,2},
\qquad\qquad
\mathbb E \|\zeta^\alpha(t)\|_{H^{0,4}_\sol}^q \le  C_{\alpha,q,4}
\end{equation}
for any $t\ge 0$. Moreover
\begin{equation}\label{lim-a}
\lim_{\alpha\to+\infty}  C_{\alpha,2}=0 , \qquad\qquad
\lim_{\alpha\to+\infty}  C_{\alpha,q,4}=0.
\end{equation}
\end{lemma}
\begin{proof}
First, we find  the estimate for $ \mathbb E \|\zeta^\alpha(t)\|^2_H$. Using inequality  \eqref{semigruppo}  and assumption {\bf (G1)} we get
\[\begin{split}
\mathbb E \|\zeta^\alpha(t)\|_{H}^2
&\le    \mathbb E\left[ \int_{0}^t
    \|e^{-(\gamma+\alpha)(t-s)}  e^{-(t-s)A}G(v(s))\|^2_{\gamma(Y;H)} ds\right]
\\
&\le   \mathbb E\left[\int_0^t
       e^{-2(\gamma+\alpha)(t-s)}\|J^{g} e^{-(t-s)A}\|_{\mathcal L(H;H)}^2
              \|J^{-g} G(v(s))\|_{\gamma(Y;H)}^2 ds\right]
\\
&\le K_{g,2}^2 \int_{0}^t e^{-2(\gamma+\alpha)(t-s)}2M^2
\left[1+\frac 1{(t-s)^g}\right] \,ds\;
\\
&=2  K_{g,2}^2 M ^{2} \int_0^{t}  e^{-2(\gamma+\alpha) r}\left[
1+ \frac 1{r^{ g }} \right]\,dr
\\ &\le 2
K_{g,2}^2 M^{2} \int_0^{\infty}  e^{-2(\gamma+\alpha) r} \left[1+
  \frac 1{r^{ g }}\right] \,dr
\end{split}\]
Calling $C_{\alpha,2}$ the  expression on the RHS, because  $g<1$ we deduce the first  limit
behaviour \eqref{lim-a} by the LDC theorem  as $\alpha \to \infty$.

The second  estimate in \eqref{lim-a}
is obtained in the same way along the lines of the proof of Lemma 3.2 in \cite{noi}. Indeed,  using again inequality  \eqref{semigruppo}  and assumption {\bf (G1)} we get
\[\begin{split}
\mathbb E &\|\zeta^\alpha(t)\|_{H^{0,4}_\sol}^q
\\
&\le  C(q) \mathbb E\left[\int_0^t e^{-(\gamma+\alpha)(t-s)}
    \|   e^{-(t-s)A}G(v(s))\|^2_{\gamma(Y;H^{0,4}_\sol)} ds\right]^{q/2}
\\
&\le  C(q) \mathbb E\left[\int_0^t e^{-(\gamma+\alpha)(t-s)}
      \|J^{g } e^{-(t-s)A}\|_{\mathcal L(H^{0,4}_\sol;H^{0,4}_\sol)}^2
              \|J^{-g} G(v(s))\|_{\gamma(Y;H^{0,4}_\sol)}^2 ds\right]^{q/2}
\\
&\le  C(q) (K_{g,4})^q \left[ \int_0^t  e^{-(\gamma+\alpha)(t-s)}
2 \big(M^2+\frac {M^2}{(t-s)^{ g}}\big) ds \right]^{q/2}.
\end{split}\]
We estimate the time integral as in the previous case to conclude the proof.
\end{proof}

By the Chebyshev inequality, from \eqref{stima-zeta-alpha} we get
\begin{equation}\label{z-bdd}
\sup_{t\ge 0} \mathbb P\left( \|\zeta^\alpha(t)\|_{H}>R\right)
\le \frac {C_{\alpha,2}}{R^2}
\end{equation}
for any $R>0$. This gives the bound \eqref{th:bdinprob} for the process $\zeta^\alpha$.

Now we look for a similar result for the process
$u^\alpha=v-\zeta^\alpha$ solving
\begin{equation}\label{eq:ualpha}
\frac {d u}{dt}(t) +Au(t) +\gamma u(t) + B(v(t),v(t))
=
\alpha\zeta^\alpha(t)+f,
\qquad u(0)=0.
\end{equation}

For this aim we need the following result
% PROP
\begin{proposition}\label{pro:bd-u}
Let $f \in H^{-1}$ and let $u^\alpha$ be the solution of
\begin{equation*}
\frac {d u}{dt}(t) +Au(t) +\gamma u(t) + B(u(t)+\zeta^\alpha(t),u(t)+\zeta^\alpha(t))
=
\alpha\zeta^\alpha(t)+f
\end{equation*}
with $u^\alpha(0)=0$ and $\zeta^\alpha$ given by \eqref{zalpha} under
the assumptions of Lemma \ref{lemma-z2}.\\
Then for any $\varepsilon >0$ there exist $\alpha, R>0$ such that
\[
\frac 1T \int_0^T \mathbb P \left(\|u^\alpha(t;0)\|_H>R\right) dt<\varepsilon
\]
for any $T>0$.
\end{proposition}
\begin{proof}
We proceed as in the proof of Proposition 3.4 in \cite{noi}.
We take the $H$-scalar product of equation \eqref{eq:ualpha} with $u^\alpha$ and get
\begin{eqnarray*}
&&\frac 12 \frac {d}{dt}\|u^\alpha(t)\|_{H}^2 +\|\nabla u^\alpha(t)\|_{L^2}^2
+\gamma \|u^\alpha(t)\|_{H}^2
\\&=&-\langle B(u^\alpha(t)+\zeta^\alpha(t), u^\alpha(t)+\zeta^\alpha(t)),u^\alpha(t)\rangle
+\alpha\langle \zeta^\alpha(t),u^\alpha(t)\rangle
+\langle f,u^\alpha(t)\rangle
\\&=&
\langle B(u^\alpha(t), u^\alpha(t)),\zeta^\alpha(t)\rangle
+\langle B(\zeta^\alpha(t), u^\alpha(t)),\zeta^\alpha(t)\rangle
+\alpha\langle \zeta^\alpha(t),u^\alpha(t)\rangle
+\langle f,u^\alpha(t)\rangle
\\&
\le& C
\|u^\alpha(t)\|^{\frac 12}_{H}\|\nabla u^\alpha(t)\|_{L^2}^{\frac 32}
\|\zeta^\alpha(t)\|_{L^4}+
  \|\nabla u^\alpha(t)\|_{L^2}\|\zeta^\alpha(t)\|_{L^4}^2
\\&&\qquad\qquad\qquad\qquad
+\alpha\|u^\alpha(t)\|_{H}\|\zeta^\alpha(t)\|_{H}
+\|u^\alpha(t)\|_{H^1} \|f\|_{H^{-1}}
\\
&
\le&
\frac 12 \|\nabla u^\alpha(t)\|_{L^2}^2
+ \frac \gamma 2 \|u^\alpha(t)\|_{H}^2+
C \|\zeta^\alpha(t)\|_{L^4}^4 \|u^\alpha(t)\|_{H}^2
+ C \|\zeta^\alpha(t)\|_{L^4}^4
\\ &&\qquad\qquad\qquad\qquad
+ C \alpha^2 \|\zeta^\alpha(t)\|_{H}^2
+C \|f\|_{H^{-1}}^2.
\end{eqnarray*}
Hence
\begin{multline*}
\frac {d}{dt}\|u^\alpha(t)\|_{H}^2\le -\gamma \|u^\alpha(t)\|_{H}^2
\\+ C_3 \left( \|\zeta^\alpha(t)\|_{L^4}^4 \|u^\alpha(t)\|_H^2
+  \|\zeta^\alpha(t)\|_{L^4}^4+ \alpha^2  \|\zeta^\alpha(t)\|_{H}^2
+ \|f\|_{H^{-1}}^2 \right)
\end{multline*}
for a constant  $C_3$ independent of $\alpha$.
Therefore, following the Da Prato-G\c{a}tarek technique from \cite{DPG},
we  infer that for every  $R>0$ the following happens
\[\begin{split}
\frac{d}{dt} &\ln(\|u^\alpha(t)\|_{H}^2 \vee R)
\\&
=1_{\{\|u^\alpha(t)\|_{H}^2>R\}} \frac1{\|u^\alpha(t)\|^2_{H}}\frac d{dt}
\|u^\alpha(t)\|_{H}^2
\\&
\le 1_{\{\|u^\alpha(t)\|_{H}^2>R\}} (-\gamma+C_3\|\zeta^\alpha(t)\|_{L^4}^4 )
\\&
\qquad +  1_{\{\|u^\alpha(t)\|_{H}^2>R\}} C_3\frac{
 \|\zeta^\alpha(t)\|_{L^4}^4+ \alpha^2 \|\zeta^\alpha(t)\|_{H}^2
+  \|f\|_{H^{-1}}^2}{\|u^\alpha(t)\|_{H}^2}
\\&
\le 1_{\{\|u^\alpha(t)\|_{H}^2>R\}} (-\gamma+C_3\|\zeta^\alpha(t)\|_{L^4}^4 )
\\&
\qquad +   \frac {C_3}R \left(
 \|\zeta^\alpha(t)\|_{L^4}^4
 + \alpha^2 \|\zeta^\alpha(t)\|_{H}^2
 +  \|f\|_{H^{-1}}^2\right).
\end{split}
\]
We integrate in time and take expectation; since $u^\alpha(0)=0$
we get that the time integral in the LHS is non negative and therefore
\[\begin{split}
\gamma \int_0^T\mathbb P\left(\|u^\alpha(t)\|_{H}^2>R\right) & dt
\le
C_3 T \sup_{t\ge0} \mathbb E [\|\zeta^\alpha(t)\|_{L^4}^4]
\\&\hspace{-2truecm}\lefteqn{+
\frac { C_3}R T\left(\sup_{t\ge 0}\mathbb E[\|\zeta^\alpha(t)\|_{L^4}^4]+
\alpha^2 \sup_{t\ge 0}\mathbb E [\|\zeta^\alpha(t)\|_{H}^2]
+\|f\|_{H^{-1}}^2\right)}
\\&
\le
T C_3 \left( C_{\alpha,4,4}+\frac {C_{\alpha,4,4}+\alpha^2 C_{\alpha,2}+\|f\|^2_{H^{-1}}}R\right).
\end{split}
\]
Now, bearing in mind \eqref{lim-a} we find that for $\alpha$ and $R$
suitably chosen the quantity
 $\dfrac 1T \displaystyle\int_0^T\mathbb
P(\|u^\alpha(t)\|_{H}^2>R)\ dt$ can be as small as we want, uniformly in time.
\end{proof}

Therefore, merging \eqref{z-bdd} and Proposition \ref{pro:bd-u}
 we get the bound \eqref{th:bdinprob} for the  process $v=\zeta^\alpha+u^\alpha$.

%%%%%%%%%%%%%%%%%%%%%%%%%%%%%%%
\section{Stationary solutions for $d=3$}
Working in the whole space, we can  prove only  the existence but not uniqueness
of martingale solutions for equation \eqref{sns}.
Hence we cannot define the Markov semigroup and  a fortiori even the
invariant measures.
But if we regularise the equation, we can prove similar results as in
$\mathbb R^2$; this way of approaching the three dimensional
Navier-Stokes equation
by regularizing the nonlinearity in order to get an equation with the
same level of difficulty as the two dimensional one,
goes back to the work of Leray \cite{Leray}. So we first
approximate equation \eqref{sns} by
\begin{equation}\label{snsm}
dv(t)+[Av(t)+\gamma v(t)+B_m\left(v(t),v(t)\right)]dt
=
G(v(t))\,dw(t)+f(t)\ dt
\end{equation}
with initial velocity $v(0)=x$.
The smoother operator $B_m$  will be defined in the next section.
We shall prove the existence of a unique solution
 for the approximating equation \eqref{snsm}.
Moreover, as in the $\mathbb R^2$ case,
there exists at least one invariant measure $\mu_m$. Considering the sequence
$\{v_m\}_{m \in \mathbb N}$ of stationary solutions of \eqref{snsm}
whose marginals at
a fixed time are $\mu_m$, we shall pass to the limit as $m\to\infty$
in order to get a stationary solution for the original equation \eqref{sns}.

%%%%%%%%%%%%%%%%%%%%%%%%%%%%%%%%%%%%%%%%%%%%%%%
\subsection{Smoothing}
In this section we investigate the smoothed equation \eqref{snsm}.
We define the regularization as follows.
For any $m>0$ let
$\rho_m(\xi)=(\frac m{2\pi})^{\frac 32} e^{-\frac m2|\xi|^2}$, $\xi
\in \mathbb R^3$,
and
\[
B_m(u,v)=B(\rho_m\ast u,v)
\]
where $\ast$ denotes the convolution. For any $1\le p <\infty$ we have
 $\|\rho_m\|_{L^p}=(\frac m{2\pi})^\frac 32 (\frac
{2\pi}{mp})^{\frac 3{2p}}$;
in particular
$\|\rho_m\|_{L^1}=1$.

By property \eqref{scambio} of the bilinear map $B$ we  have
\begin{equation}\label{scambio-m}
\langle B_m(u,v),z\rangle =-\langle B_m(u,z),v\rangle , \qquad
\langle B_m(u,v),v\rangle=0
\end{equation}
for any $ u, v,z \in H^1$.
In addition the following estimates hold.
\begin{lemma}
For any $m>0$ we have
\begin{eqnarray}\label{stimaBm-L4}
\|B_m(u,v)\|_{H^{-1}}&\le& \|u\|_{L^4}\|v\|_{L^4}
\\ \label{laCM}
\|B_m(u,v)\|_{H^{-1}}&\le& \|\rho_m\|_{L^2} \|u\|_{H}\|v\|_{H}
\\ \label{Bm1}
\|B_m(u,v)\|_{H^{-1-g}} &\le &C \|\rho_m\|_{L^{\frac 6{4+g}}} \|u\|_{H} \|v\|_{H^{\frac{1-g}2}}
\\ \label{Bm2}
\|B_m(u,v)\|_{H^{-1-g}}&\le&  C \|\rho_m\|_{L^{\frac 6{4+g}}} \|u\|_{H^{\frac{1-g}2}} \|v\|_{H}
\end{eqnarray}
\end{lemma}
\begin{proof}
We use repeatedly Young and Sobolev inequalities and that
\[
\|B(u,v)\|_{H^{-a}}=\sup_{\|\phi\|_{H^a}\le 1}|\langle B(u,v),\phi\rangle|
=\sup_{\|\phi\|_{H^a}\le 1}|\langle B(u,\phi),v\rangle|
\]
for smooth enough vectors. Then, by density,  the same holds for the regularity
involved at each step.

We prove \eqref{stimaBm-L4}, which is the only uniform estimate.
\[\begin{split}
|\langle B_m(u,\phi),v\rangle|
&=|\langle B(\rho_m\ast u,\phi),v\rangle|
\\&
\le \|\rho_m\ast u\|_{L^4}\|\nabla \phi\|_{L^2}\|v\|_{L^4}
\\&
\le \|\rho_m\|_{L^1}\|u\|_{L^4}\|\phi\|_{H^1}\|v\|_{L^4}
\end{split}\]

We prove \eqref{laCM}.
\[
\begin{split}
|\langle B_m(u,\phi),v\rangle|
&=|\langle B(\rho_m\ast u,\phi),v\rangle|
\\&
\le \|\rho_m\ast u\|_{L^\infty}\|\nabla \phi\|_{L^2}\|v\|_{L^2}
\\&
\le \|\rho_m\|_{L^2}\|u\|_{L^2}\|\phi\|_{H^1}\|v\|_{L^2}
\end{split}\]

We prove  \eqref{Bm1}
\[
\begin{split}
|\langle B_m(u,\phi),v\rangle|
&=|\langle B(\rho_m\ast u,\phi), v\rangle|\\
&\le \|\rho_m\ast u\|_{L^{\frac 6{1+g}}} \|\nabla\phi\|_{L^{\frac 6{3-2g}}} \|v\|_{L^{\frac 6{2+g}}}
\\
& \le C \|\rho_m\|_{L^{\frac 6{4+g}}} \|u\|_{L^2} \|\nabla\phi\|_{H^g} \|v\|_{H^{\frac {1-g}2}}
\end{split}
\]
and similarly  \eqref{Bm2}
\[
\begin{split}
|\langle B_m(u,\phi),v\rangle|
&=|\langle B(\rho_m\ast u,\phi), v\rangle|
\\
&\le \|\rho_m\ast u\|_{L^{\frac 3g}} \|\nabla\phi\|_{L^{\frac 6{3-2g}}} \|v\|_{L^2}
\\
& \le C \|\rho_m\|_{L^{\frac 6{4+g}}} \|u\|_{L^{\frac 6{2+g}}} \|\nabla\phi\|_{H^g} \|v\|_{L^2}
\\
& \le C \|\rho_m\|_{L^{\frac 6{4+g}}} \|u\|_{H^{\frac {1-g}2}} \|\phi\|_{H^{1+g}} \|v\|_{H} \end{split}
\]
\end{proof}

In the limit as $m \to \infty$  we recover the operator $B$.
Bearing in mind \eqref{stimaB-a} with
 given  $a>\frac 52$, we get for all  $u,v \in H$
\begin{equation}\label{lim-Bm} 
\|B_m(u,v)-B(u,v)\|_{H^{-a}}=\|B(\rho_m\ast u-u,v)\|_{H^{-a}}
\le C  \|\rho_m \ast u-u\|_{L^2} \|v\|_{L^2}.
\end{equation}
Note that the RHS in \eqref{lim-Bm} converges to $0$ as $m\to \infty$. Indeed,
 by the Plancherel equality
\[
\|\rho_m \ast u-u\|_{L^2}=\|\hat \rho_m\hat u-\hat u\|_{L^2}
\]
with the Fourier transform
 $\hat \rho_m(\xi)=(2\pi)^{-\frac 32} e^{-\frac {|\xi|^2}{2m}}
\to 1$ pointwise as $m\to\infty$ and $\|\hat \rho_m\hat u\|_{L^2}\le
(2\pi)^{-\frac 32}\|\hat u\|_{L^2}$. Hence,  $\hat \rho_m\hat u-\hat u\to 0$
pointwise and,
by dominated convergence,
$\|\hat \rho_m \hat u-\hat u\|_{L^2}\to 0 $ for any given $u\in L^2$.

Here is our first result on the smoothed equation for any $m>0$.
It involves the norm
\[
\n{v}_T=\|v\|_{L^\infty(0,T;H)}+\|v\|_{L^2(0,T;H^\delta)}+
 \|v\|_{L^{\frac 83}(0,T;L^4)}+\|v\|_{C^\beta([0,T];H^{-1})}
\]
with $ \beta \in (0,\frac 14]$ and $\delta\in (0,1]$
 such that $\beta+\frac \delta 2<\frac {1-g}2$.
% THEOREM %
\begin{theorem}\label{teo-snsm}
Let $d=3$.
If $x \in H$,  $f \in L^p_{\mathrm{loc}}([0,\infty);H^{-1})$ for some $p>2$ and
assumptions {\bf (G1)}-{\bf (G2)}-{\bf (G3)}-{\bf (G4)} are satisfied,
then   there exists a unique  solution $v_m$ of equation \eqref{snsm}
on the time interval $[0,\infty)$ with initial velocity $x$;
in addition, there exist $\beta \in (0,\frac 14]$ and $\delta\in (0,1]$
 with  $\beta+\frac \delta 2<\frac {1-g}2$ such that $\mathbb P$-a.s
\[
v_m\in  C([0,\infty);H)\cap L^{\frac 83}_{\mathrm{loc}}([0,\infty);L^4)\cap L^2_{\mathrm{loc}}([0,\infty);H^\delta)\cap
C^\beta_{\mathrm{loc}}([0,\infty);H^{-1})
.
\]
and, for each $T>0$,
\begin{multline}\label{stime-vm-per-tight}
\n{v_m}_T^2
\le C_4 \Big[ (1+T^2) \big(1
                    + \Psi(z_m,T)^2 (1+\Phi(z_m,T)^2) e ^{2\Phi(z_m,T)} \big)
        \\  +1+\|z_m\|^4_{L^{\frac 83}(0,T;L^4)}+(1+T)
         \|z_m\|^2_{C^\beta([0,T];H^\delta)} +T^{\frac 12} \|f\|^2_{L^2(0,T;H^{-1})} \Big]
\end{multline}
$\mathbb P$-a.s., where
\begin{eqnarray}
\Psi(z_m,T)&=&\|x\|_H^2+C_5\|z_m\|_{L^{4}(0,T;L^4)}^4+C_5 \|f\|_{L^2(0,T;H^{-1})}^2,
\label{lapsi}
\\
\Phi(z_m,T)&=&C_6\|z_m\|_{L^{8}(0,T;L^4)}^8,
\label{laphi}
\end{eqnarray}
the positive constants $C_4, C_5, C_6$ are independent of $m$ and
$T$, and
the process $z_m$ is the solution of equation
\begin{equation}\label{eq:zm}
\left\{
\begin{aligned}
dz(t)&+Az(t)\ dt+\gamma z(t)\ dt=G(v_m(t))\,dw(t),\;t>0;\\ z(0)&=0.
\end{aligned}
\right.
\end{equation}
\end{theorem}
\begin{proof}
The proof is based on our previous paper \cite{noi}. We present a proof on a fixed time interval $[0,T]$, see however Remark \ref{rem-th:esist2}.
 First we prove the existence of  a martingale solution; then the pathwise uniqueness.
Hence by a result of
\cite{IW} there exists a unique strong solution  in the probabilistic
sense. More precisely,
the existence of  a martingale solution  is obtained with the procedure
used in \cite{noi} for the 3d stochastic Navier-Stokes equation. More
regularity and pathwise uniqueness come from the techniques used in
\cite{noi} for the 2d stochastic Navier-Stokes equation.

We provide some details for the reader's convenience.

As usual, we set $v_m=z_m+u_m$ with
\[
z_m(t)=\int_0^te^{-\gamma(t-s)}e^{-(t-s)A}G(v_m(s))\,dw(s)
\]
solving equation \eqref{eq:zm}
and $u_m$ solving
\begin{equation}\label{eq:um}
\left\{
\begin{aligned}
\frac {d u_m}{dt}(t) &+Au_m(t)+\gamma u_m(t) +B_m(v_m(t),v_m(t))=f(t),\;
t\in (0,T];\\
 u_m(0)&=x.
\end{aligned}
 \right.
\end{equation}

Keeping in mind \eqref{stimaBm-L4} and \eqref{scambio-m} we have that
for any  finite $m$ the operator $B_m$ enjoys the same properties
as $B$ necessary  to prove the existence of martingale solutions for the
three dimensional stochastic Navier-Stokes equation \eqref{sns};
the additional damping term
$\gamma \|u_m(t)\|_H^2$ appearing in the energy estimate
has no effect on the apriori estimates as we have noticed  in the previous
section, see \eqref{energyest}.
Hence, assuming {\bf (G1)}-{\bf (G2)}-{\bf (G3)} we obtain the same
result as Theorem 3.6 in  \cite{noi}:
there exists a
martingale solution to equation \eqref{snsm}, and a.e.
 path of this process $v_m$  is
in $L^\infty(0,T;H)\cap C([0,T];H_{\mathrm{w}})\cap L^{\frac 83}(0,T;L^4)$.

Let us show that
a.e.  path of this solution process is in
$ C([0,T];H)$; here we need the  smoothing in order to proceed as in
the two dimensional case. We work pathwise. First,
we have $z_m\in C([0,T];H)$  as in Lemma 3.3 of
\cite{noi}.
Since  $v_m\in
L^\infty(0,T;H)$,   \eqref{laCM} implies  that $B_m(v_m(t),v_m(t))
\in L^\infty(0,T;H^{-1})$.
Moreover, $u_m \in L^2(0,T;H^1)$.
Hence
\[
\frac {d u_m}{dt}(t) =-Au_m(t)-\gamma u_m(t) -B_m(v_m(t),v_m(t))+f(t)\in L^2(0,T;H^{-1}).
\]
We conclude  by means of  a classical result, see Ch III Lemma 1.2 of \cite{temam}:
if
$u_m \in L^2(0,T;H^1)$ and $\frac {d u_m}{dt}\in
L^2(0,T;H^{-1})$  then
$u_m \in C([0,T];H)$. Therefore $v_m=z_m+u_m \in C([0,T];H)$ a.s..

Let us prove the estimate \eqref{stime-vm-per-tight}.
Taking the $H$-scalar product of  equation \eqref{eq:um}  with $u_m$ and using
the bilinearity of $B_m$ and \eqref{scambio-m}, we get
\begin{multline}\label{ener-um}
\frac 12 \frac {d}{dt}\|u_m(t)\|_{H}^2 +\|\nabla u_m(t)\|_{L^2}^2
+\gamma \|u_m(t)\|_{H}^2
\\=-\langle B_m(u_m(t)+ z_m(t),u_m(t)+z_m(t)),u_m(t)\rangle
+\langle f(t),u_m(t)\rangle.
\end{multline}
Using \eqref{stimaBm-L4}, we estimate the trilinear term
 as in the proof of Proposition 3.4 in \cite{noi}:
\begin{equation}\label{sti-B}
-\langle B_m(u_m+z_m,u_m+z_m), u_m\rangle
\le \frac 14 \|\nabla u_m\|_{L^2}^2
   +C \|u_m\|_{H}^2 \|z_m\|_{L^4}^8+  C \|z_m\|_{L^4}^4
\end{equation}
for some positive  constant $C$ independent of $m$.
Moreover
\[
|\langle f,u_m\rangle|
\le \|f\|_{H^{-1}}(\|u_m\|_{L^2}+\|\nabla u_m\|_{L^2})
\le  \gamma \|u_m\|_{H}^2+ \frac 14 \|\nabla u_m\|_{L^2}^2
+ (\frac1{4\gamma}+1) \|f\|_{H^{-1}}^2.
\]
Inserting these estimates in \eqref{ener-um} we get
\begin{multline}\label{stima x um}
 \frac {d}{dt}\|u_m(t)\|_{H}^2+ \|\nabla u_m(t)\|_{L^2}^2
\\\le
C_5\|z_m(t)\|_{L^4}^4+C_5 \|f(t)\|_{H^{-1}}^2
+C_6\|u_m(t)\|_{H}^2 \|z_m(t)\|_{L^4}^8
\end{multline}
for some constants $C_5, C_6$ independent of $T$ and $m$.
Gronwall Lemma applied to
\[
 \frac {d}{dt}\|u_m(t)\|_{H}^2\le C_6\|z_m(t)\|_{L^4}^8 \|u_m(t)\|_{H}^2
+C_5\|z_m(t)\|_{L^4}^4+C_5 \|f(t)\|_{H^{-1}}^2
\]
gives
\begin{equation}\label{stimaH}
\begin{split}
\sup_{0\le t\le T} \|u_m(t)\|_{H}^2 &\le
\|x\|_H^2 e^{\int_0^T C_6 \|z_m(r)\|^8_{L^4}dr}
\\ & \qquad +C_5
 \int_0^T e^{\int_s^T  C_6 \|z_m(r)\|^8_{L^4} dr}
\left(\|z_m(s)\|^4_{L^4}+\|f(s)\|^2_{H^{-1}}\right)ds\\
\\&\le
\Psi(z_m,T)e^{\Phi(z_m,T)}
\end{split}
\end{equation}
where $\Psi$ and $\Phi$ are defined in \eqref{lapsi} and \eqref{laphi}, respectively.
Integrating in time \eqref{stima x um} we get
\begin{equation}\label{stimaH1}
\begin{split}
\int_0^T \|\nabla u_m(t)\|_{L^2}^2 dt
&\le
\Psi(z_m,T)+\Phi(z_m,T)\|u_m\|_{L^\infty(0,T;H)}^2
\\& \le
\Psi(z_m,T)+\Phi(z_m,T)\Psi(z_m,T)e^{\Phi(z_m,T)}.
\end{split}
\end{equation}

Now, we recall the continuous embedding
$H^{1,\frac 43}(0,T;H^{-1})\subset C^{\frac 14}([0,T];H^{-1})$; using
\eqref{stimaBm-L4} we proceed as
 in Proposition 3.4 of \cite{noi}
to get
\begin{equation}\label{stimaC14}
\begin{split}
\|u_m\|&_{C^{\frac 14}([0,T];H^{-1})}\\
&\le C \|u_m\|_{H^{1,\frac 43}(0,T;H^{-1})}\\
&=C (\|u_m\|_{L^{\frac 43}(0,T;H^{-1})}
  +\|-A u_m-\gamma u_m -B_m(v_m,v_m)+f\|_{L^{\frac 43}(0,T;H^{-1})})
\\
&\le
C \Big(\|u_m\|_{L^{\frac 43}(0,T;H)}
  +\|u_m\|_{L^{\frac 43}(0,T;H^1)}+\gamma \|u_m\|_{L^{\frac 43}(0,T;H)}
\\&\qquad
    + \|B_m(u_m+z_m,u_m+z_m)\|_{L^{\frac 43}(0,T;H^{-1})} +\|f\|_{L^{\frac 43}(0,T;H^{-1})}\Big)
\\
&\le C\Big((1+\gamma)T^{\frac 34}\|u_m\|_{L^\infty(0,T;H)}
 + T^{\frac 14}\|u_m\|_{L^2(0,T;H^1)}
 + \|u_m\|^2_{L^{\frac 83}(0,T;L^4)}
\\&\qquad
  + \|z_m\|^2_{L^{\frac 83}(0,T;L^4)}+T^{\frac 14} \|f\|_{L^2(0,T;H^{-1})} \Big).
\end{split}
\end{equation}
The  Gagliardo-Nirenberg inequality
\[
\|u_m(t)\|_{L^4}\le C
\|u_m(t)\|_{L^2}^{\frac 14} \|\nabla u_m(t)\|_{L^2}^{\frac 34}
\]
gives
\begin{equation}\label{stimaGN}
\begin{split}
\|u_m\|_{L^{\frac 83}(0,T;L^4)}^2
&\le
C \|u_m\|_{L^\infty(0,T;L^2)}^{\frac 12} \|\nabla
u_m\|_{L^2(0,T;L^2)}^{\frac 32}
\\
&
\le
C(\|u_m\|_{L^\infty(0,T;H)}^2+ \|\nabla u_m\|^2_{L^2(0,T;H)} ).
\end{split}\end{equation}

Summing up, we can bound $\n{ u_m}_T$.
First,
\begin{eqnarray*}
\n{ u_m}_T
&\le&
C\Bigl((\|u_m\|_{L^\infty(0,T;H)} + \|u_m\|_{L^2(0,T;H^1)}
+ \|u_m\|_{L^{\frac 83}(0,T;L^4)}\\
&+&(1+T^{\frac 14-\beta}) \|u_m\|_{C^{\frac 14}([0,T];H^{-1})}
)\Bigr);
\end{eqnarray*}
then by \eqref{stimaH}, \eqref{stimaH1}, \eqref{stimaC14} and \eqref{stimaGN} we get
\[\begin{split}
\n{ u_m}_T ^2
&\le
C\Big((1+T^2)\|u_m\|^2_{L^\infty(0,T;H)}
 +(1+T)\|\nabla u_m\|^2_{L^2(0,T;L^2)} +\|u_m\|^4_{L^\infty(0,T;H)}
\\ &
\qquad+\|\nabla u_m\|^4_{L^2(0,T;L^2)}
 + \|z_m\|^4_{L^{\frac 83}(0,T;L^4)}+T^{\frac 12} \|f\|^2_{L^2(0,T;H^{-1})}
\Big)\\
&\le
C\Big( (1+T^2)[1+\Psi(z_m,T) e^{2\Phi(z_m,T)}(1+\Phi(z_m,T)^2)]
\\&
\qquad +\|z_m\|^4_{L^{\frac 83}(0,T;L^4)}+T^{\frac 12} \|f\|^2_{L^2(0,T;H^{-1})}
\Big)
\end{split}
\]
for some positive constant $C$ independent of $T$ and $m$.

For the process $z_m$ we have
\[
\n{ z_m}_T\le C \left( (1+T^{\frac 12})\|z_m\|_{C^\beta([0,T];H^\delta)}+\|z_m\|_{L^{\frac 83}(0,T;L^4)}
\right)
\]
for some positive constant $C$ independent of $T$ and $m$.

We finally
 get \eqref{stime-vm-per-tight} for $v_m=u_m+z_m$.

As far as the pathwise uniqueness is concerned,
we proceed in a similar way we did in  Theorem 4.1 of \cite{noi}
 for the two dimensional stochastic Navier-Stokes equation.
Let $v_m$ and $\tilde v_m$ be two
processes solving
\eqref{snsm} on the same stochastic basis
with the same Wiener process,  initial velocity and deterministic force $f$.
Set $V=v_m-\tilde v_m$; this difference satisfies
\begin{multline*}
dV(t) +[AV(t) +\gamma V(t)
+B_m(v_m(t),v_m(t))-B_m(\tilde v_m(t),\tilde v_m(t))]\ dt
\\=[G(v_m(t))-G(\tilde v_m(t))]\,dw(t)
\end{multline*}
with $V(0)=0$; the equation  is equivalent to
\begin{eqnarray*}
dV(t) &+&[AV(t) +\gamma V(t) + B_m(V(t),v_m(t))+B_m(\tilde v_m(t),V(t))]\ dt
\\&=&[G(v_m(t))-G(\tilde v_m(t))]\,dw(t).
\end{eqnarray*}

We will apply the  It\^o formula for the process
$e^{-\int_0^t \sigma(s)ds}\|V(t)\|^2_{H^{-g}}$, $t \in [0,T]$, 
by choosing $\sigma$ as it has been done in paper \cite{S} by B. Schmalfuss:
\[
\sigma(s)=L_g^2+2M_{m,\gamma} \left(\|v_m(s)\|_{H}
+\|\tilde v_m(s)\|_{H}\right)^{\frac4{1-g}}
\]
with $L_g$ the Lipschitz constant given in {\bf (G4)}
and $M_{m,\gamma}$ the constant appearing later on in \eqref{cxunic}.
We have  $\sigma\in L^1(0,T)$ $\mathbb P$-a.s.
since $v_m, \tilde v_m\in C([0,T;H)$ $\mathbb P$-a.s..\\
First
\[
d \left(e^{-\int_0^t \sigma(s)ds}\|V(t)\|^2_{H^{-g}}\right)=-  \sigma(t) e^{-\int_0^t \sigma(s)ds}\|V(t)\|^2_{H^{-g}} dt +
e^{-\int_0^t \sigma(s)ds}d\|V(t)\|^2_{H^{-g}}
\]
and the latter differential is well defined and given by
\[\begin{split}
\frac 12 d\|V(t)\|^2_{H^{-g}}=&-\|\nabla V(t)\|_{H^{-g}}^2dt-\gamma\|V(t)\|_{H^{-g}}^2dt
\\&
-\langle J^{-1-g}[B_m(V(t),v_m(t))+B_m(\tilde v_m(t),V(t))],J^{1-g}V(t)\rangle dt
\\
&+\langle J^{-g}[G(v_m(t))-G(\tilde v_m(t))]\,dw(t), J^{-g}V(t)\rangle
\\&
+\frac 12 \|G(v_m(t))-G(\tilde v_m(t))\|_{\gamma(Y;H^{-g})}^2 dt.
\end{split}
\]

We deal with the trilinear terms:
\begin{equation}\label{cxunic}
    \begin{split}
|\langle J^{-1-g}&[B_m(V,v_m)+B_m(\tilde v_m,V)],J^{1-g}V\rangle|
\\
&\le (\|B_m(V,v_m)\|_{H^{-1-g}}+\|B_m(\tilde v_m,V)\|_{H^{-1-g}})\|V\|_{H^{1-g}}
\\
&\le C \|\rho_m\|_{L^{\frac 6{4+g}}} (\|v_m\|_H+\|\tilde v_m\|_H)\|V\|_{H^{\frac{1-g}2}}\|V\|_{H^{1-g}}
\;\text{ by }  \eqref{Bm1}-\eqref{Bm2}
\\
& \le C  \|\rho_m\|_{L^{\frac 6{4+g}}} (\|v_m\|_H+\|\tilde v_m\|_H)
\|V\|_{H^{-g}}^{\frac{1-g}2}\|V\|_{H^{1-g}}^{\frac{3+g}2}
\;\text{ by interpolation}
\\
&\le \tfrac {\min(1,\gamma)}2  \|V\|_{H^{1-g}}^2
 +M_{m,\gamma} (\|v_m\|_H+\|\tilde v_m\|_H)^{\frac4{1-g}} \|V\|^2_{H^{-g}} \;
\text{ by Young inequality}
 \\
&=   \tfrac {\min(1,\gamma)}2 \|\nabla V\|_{H^{-g}}^2+ \tfrac {\min(1,\gamma)}2\|V\|_{H^{-g}}^2
 +M_{m,\gamma}  (\|v_m\|_H+\|\tilde v_m\|_H)^{\frac4{1-g}} \|V\|^2_{H^{-g}}.
\end{split}
\end{equation}

Therefore, using {\bf (G4)}   we get
\begin{multline}\label{stima:Sch}
d \left(e^{-\int_0^t \sigma(s)ds}\|V(t)\|^2_{H^{-g}}\right)
\\\le
e^{-\int_0^t \sigma(s)ds}\langle J^{-g}[G(v_m(t))-G(\tilde v_m(t))]\,dw(t), J^{-g}V(t)\rangle.
\end{multline}
The RHS is  a local martingale; indeed if we define the stopping time
\[
\tau_N=T\wedge \inf\{t \in [0,T]:\|V(t)\|_{H^{-g}}>N\}
\]
and
\[
M_N(t)=\int_0^{t\wedge \tau_N} e^{-\int_0^r \sigma(s)ds}
\langle J^{-g}V(r), J^{-g}[G(v_m(r))-G(\tilde v_m(r))]\,dw(r)\rangle
\]
then
\begin{align*}
\mathbb E [M_N(t)^2]&\le \mathbb E \int_0^{t\wedge \tau_N}
  e^{-2\int_0^r \sigma(s)ds}\|V(r)\|_{H^{-g}}^2\|G(v_m(r))-G(\tilde v_m(r))\|^2_{\gamma(Y;H^{-g})}dr\\
&\le L_g^2 \mathbb E \int_0^{t\wedge \tau_N}\|V(r)\|_{H^{-g}}^4 dr \le L_g^2 N^4 t .
\end{align*}
Hence, $M_N$ is a square integable martingale;
in particular $\mathbb E [M_N(t)]=0$ for any $t$.

Therefore, by integrating  \eqref{stima:Sch} over  $[0,t\wedge\tau_N]$
and taking  the expectation we get
\[
\mathbb E \left[e^{-\int_0^{t\wedge \tau_N} \sigma(s)ds}\|V(t\wedge \tau_N)\|^2_{H^{-g}}\right]
\le 0.
\]
So
\[
e^{-\int_0^{t\wedge \tau_N} \sigma(s)ds}\|V(t\wedge \tau_N)\|^2_{H^{-g}}=0 \qquad \mathbb P-a.s.
\]
Since $\displaystyle\lim_{N\to \infty} \tau_N=T\ \mathbb P$-a.s., we get in the limit that for any $t\in [0,T]$
\[
e^{-\int_0^{t} \sigma(s)ds}\|V(t)\|^2_{H^{-g}}=0 \qquad \mathbb P-a.s.
\]
Thus, if we take a  sequence $\{t_k\}_{k=1}^\infty$ which is dense in $[0,T]$ we have
\[
\mathbb P\left(\|V(t_k)\|_{H^{-g}}=0 \; \text{for all } k \in \mathbb N\right)=1.
\]
Since a.e. path of the process $V$ belongs to $ C([0,T];H)\subset C([0,T]; H^{-g})$, we
get
\[
\mathbb P\left( \|V(t)\|_{H^{-g}}=0 \; \text{for all } t\right) =1.
\]
This gives pathwise uniqueness.
\end{proof}

Now, for each $m\in \mathbb N$  we define the Markov semigroup
$\{P^{(m)}_t\}_{t\ge 0}$ associated to \eqref{snsm} as
\[
P^{(m)}_t\phi(x)= {\mathbb E}[\phi( v_m(t;x))]
\]
for any  bounded Borel function $\phi:H\to \mathbb R$.

Let us point out that the  result of Theorem \ref{teo-snsm} holds also  for general
initial distribution, see
Corollary 22 in \cite{On2005} (based on the finite dimensional
case considered in \cite{IW}).
Therefore, we have the following
% COROLLARY
\begin{corollary}\label{cor-snsm}
Let $d=3$. If $\mu$ is a probability measure on the Borelian subsets of $H$,
$f \in L^p(0,T;H^{-1})$ for some $p>2$ and assumptions
{\bf (G1)}-{\bf (G2)}-{\bf (G3)}-{\bf (G4)} are satisfied,
then   there exists a unique  solution $v_m$ of equation \eqref{snsm}
on the time interval $[0,T]$ with initial velocity of law $\mu$ and
the same properties given in  Theorem \ref{teo-snsm} hold.
\end{corollary}

As in the previous section, by means of
Maslowski and Seidler result of Theorem \ref{MSprop}, we can prove the existence of at least one
invariant measure for the smoothed equation \eqref{snsm}.
% THEOREM
\begin{theorem}\label{th:mis-inv-3d-m}
Let $d=3$. If $f \in H^{-1}$ and assumptions
{\bf (G1)}-{\bf (G2)}-{\bf (G3)}-{\bf (G4)} are satisfied,
then there exists at least one invariant measure $\mu_m$ for equation
\eqref{snsm} and $\mu_m(H)=1$.
\end{theorem}
\begin{proof}
The proof is the same as that of Theorem \ref{th:mis-inv-2d}. Indeed,
both the sequentially weakly Feller property
and the boundedness in probability result
are based on the results of
\cite{noi}, which hold for the stochastic damped Navier-Stokes equation \eqref{sns} as
well as for the smoothed version \eqref{snsm}.
In particular the estimates for the boundedness in probability
of the sequence of $v_m=\zeta^\alpha_m+u^\alpha_m$ come from those
for $\zeta^\alpha_m$ and $u^\alpha_m$, where
\[
\zeta_m^\alpha(t)=\int_0^t e^{-(\gamma+\alpha)(t-s)}e^{-(t-s)A}G(v_m(s))\,dw(s), \;\; t\in [0,T],
\]
and
\begin{equation}\label{uma}
\left\{
\begin{aligned}
\frac {d u_m^\alpha}{dt}(t) &+Au_m^\alpha((t) +\gamma u_m^\alpha((t)
+ B_m(v_m(t),v_m(t))
=
\alpha\zeta^\alpha_m(t)+f,\\
 u_m^\alpha(0)&=0.
\end{aligned}
\right.
\end{equation}
In particular,
Lemma \ref{lemma-z2}  gives
\[
\sup_{t\ge 0} \mathbb E \|\zeta^\alpha_m(t)\|_H^2 \le  C_{\alpha,2},
\qquad\qquad
\sup_{t\ge 0} \mathbb E \|\zeta^\alpha_m(t)\|_{H^{0,4}_\sol}^q \le  C_{\alpha,q,4}
\]
(here we need $q=2$ and $q=4$ and by the way we notice that the
estimates are uniform in $m$).
As far as the estimate  for $u^\alpha_m$ is concerned, we proceed  as
follows.
By taking the $H$-scalar product of equation \eqref{uma} with $u^\alpha_m(t)$ we
get
\[
\begin{split}
\frac 12 \frac {d}{dt}\|u^\alpha_m(t)\|_{H}^2
 &+ \|\nabla u^\alpha_m(t)\|^2_{L^2}+\gamma \|u^\alpha_m(t)\|_{H}^2
\\
&=-\langle
B_m(u^\alpha_m(t)+\zeta^\alpha_m(t)m,u^\alpha_m(t)+\zeta^\alpha_m(t)), u^\alpha_m(t)\rangle
\\
&\qquad +\alpha \langle \zeta^\alpha_m(t), u^\alpha_m(t)\rangle
+\langle f, u^\alpha_m(t)\rangle.
\end{split}
\]
We estimate the trilinear
term as in
 \eqref{sti-B} (but with different constants)
and get
\begin{multline*}
\frac {d}{dt}\|u^\alpha_m(t)\|_{H}^2\le -\gamma \|u^\alpha_m(t)\|_{H}^2
+ C_7\|\zeta^\alpha_m(t)\|_{L^4}^8
\|u^\alpha_m(t)\|_{H}^2\\
+ \alpha^2 C_7 \|\zeta^\alpha_m(t)\|_{H}^2
+ C_7 \|\zeta^\alpha_m(t)\|_{L^4}^4+C_7 \|f\|_{H^{-1}}^2
\end{multline*}
providing, as in the proof of Theorem \ref{pro:bd-u}, that
\[\begin{split}
\gamma\frac 1T  \int_0^T\mathbb P &\left(\|u^\alpha_m(t)\|_{H}^2>R\right) dt
\le
C_7 \sup_{t\ge0} \mathbb E [\|\zeta^\alpha_m(t)\|_{L^4}^8]
\\&
+
\frac 1R\ C_7\left(\alpha^2 \sup_{t\ge 0}\mathbb E [\|\zeta^\alpha_m(t)\|_{H}^2]
+\sup_{t\ge 0}\mathbb E[\|\zeta^\alpha_m(t)\|_{L^4}^4]+
\|f\|_{H^{-1}}^2\right)
\\&
\le C_7\left( C_{\alpha,8,4}
   +\frac {\alpha^2 C_{\alpha,2}+C_{\alpha,4,4}+\|f\|^2_{H^{-1}}}R \right).
\end{split}\]
From here we conclude as in the two dimensional case considered in the
previous section, since  the RHS above can be as small as we
want for a suitable choice of $\alpha$ and $R$.
\end{proof}

%%%%%%%%%%%%%%%%%%%%%%%%%%%%%%%%%%%%%%%%%%%%%%%%%%%%
\subsection{Existence of stationary martingale solutions}
Now we go back to the original equation \eqref{sns}.
For each $m\in \mathbb N$,  we fix an invariant measure $\mu_m$ for
the smoothed equation
\eqref{snsm} as given by Theorem \ref{th:mis-inv-3d-m}.
Therefore we denote by $v_m$   the stationary
solution of \eqref{snsm} whose
marginal at any fixed time is $\mu_m$; this is the solution given in
Corollary \ref{cor-snsm} with initial
velocity of  law $\mu_m$, given a   probability space  rich enough to
support a random variable with law $\mu_m$.
In the limit as $m \to \infty$, we shall get a stationary solution of equation \eqref{sns}.

First, we prove a tightness result.
%PROPOSITION
\begin{proposition} \label{pro-tight}
Let $d=3$,  $f \in H^{-1}$ and assume
{\bf (G1)}-{\bf (G2)}-{\bf (G3)}-{\bf (G4)}.
For any $m \in \mathbb N$, let $\mu_m$ be an invariant measure
for equation \eqref{snsm} as given in Theorem  \ref{th:mis-inv-3d-m}.
Then
the sequence of the laws of the stationary processes $v_m$,  solving equation \eqref{snsm} with
initial velocity of law  $\mu_m$,  is tight in the space
\[
Z=C(\mathbb R_+;H_{\mathrm{w}})\cap L^2_{\mathrm{loc}}(\mathbb R_+;H_{\mathrm{loc}})
\cap
\Big(L^{\frac 83}_{\mathrm{loc}}(\mathbb R_+;L^4)\Big)_{\mathrm{w}}
\cap C(\mathbb R_+;U^\prime)
\]
with the topology $\mathcal T$  given by the supremum of the corresponding topologies.
\end{proposition}
\begin{proof}
We have to prove that for every $\eps>0$ we can find a compact subset
$ K_\eps$ of $Z$
such that
\[
\sup_m \mathbb P\left( v_m \notin K_\eps\right) < \eps.
\]

Let us fix $\eps>0$. Choose
$ \beta \in (0,\frac 14]$ and $\delta\in (0,1]$
 such that $\beta+\frac \delta 2<\frac {1-g}2$.
We shall take a  set of the form
$K(\alpha^{(\eps)})=\{v \in Z:\n{ v}_N \le \alpha_N^{(\eps)} \text{
  for any } N \in \mathbb N\}$ (which is compact
according to Lemma \ref{lemma-Zcomp}) with
\[
\n{ v}_N=\|v\|_{L^\infty(0,N;H)}+\|v\|_{L^2(0,N;H^\delta)}+
 \|v\|_{L^{\frac 83}(0,N;L^4)}+\|v\|_{C^\beta([0,N];H^{-1})}.
\]

As in \cite{BO2011}, it is sufficient to find a
sequence $\alpha^{(\eps)}=(\alpha^{(\eps)}_1,\alpha^{(\eps)}_2,\ldots)$ such that
\begin{equation}\label{succ}
\sup_m \mathbb P\left( \n{ v_m}_N>\alpha_N^{(\eps)}\right) < \frac 1{2^N}\eps.
\end{equation}
Indeed, this implies tightness, since
\begin{eqnarray*}
\sup_m \mathbb P\left( v_m \notin K(\alpha^{(\eps)})\right) &=&
\sup_m \mathbb P\left( \bigcup_{N=1}^\infty \{\n{ v_m}_N>\alpha_N^{(\eps)}\}\right)
\\
&\le& \sum_{N=1}^\infty \sup_m \mathbb P\left( \n{ v_m}_N>\alpha_N^{(\eps)}\right)
< \eps.
\end{eqnarray*}

Bearing in mind \eqref{stime-vm-per-tight}
we will estimate $\mathbb P\left( \n{ v_m}_N>\alpha_N^{(\eps)}\right) $.
From Lemma \ref{lemma-z2} we infer that there exists a constant
$C_{8}$ depending on $p,g,\gamma,K_{g,4}$
such that for all  $m$ and $T$
\begin{equation}\label{z-in-Lp}
\mathbb E \|z_m\|_{L^p(0,T;L^4)}^p\le C_8(1+T).
\end{equation}
From the proof of Lemma 3.3 of \cite{noi} we infer that there exists a constant
$C_9$ depending on $p, \beta,\delta, K_{g,2}$
such that for all  $m$ and $T$
\begin{equation}\label{z-in-C}
\mathbb E \|z_m \|_{C^\beta([0,T];H^\delta)}
\le C_9(1+T^{\frac{1-g}2-\beta-\frac \delta 2}).
\end{equation}
Therefore, by the Chebyshev inequality for any  $b>0$  and $m \in \mathbb N$
we have
\begin{equation}\label{Cheb-z-in-Lp}
\mathbb P\left( \|z_m\|_{L^p(0,T;L^4)}^p>b\right) \le \frac {C_8(1+T)}b
\end{equation}
\begin{equation}
\mathbb P\left( \|z_m\|_{C^\beta([0,T];H^\delta)}>b\right) \le
\frac {C_9(1+T^{\frac{1-g}2-\beta-\frac \delta 2})}b .
\end{equation}

Bearing in mind the definition of $\Phi(z_m,N)$ given in
\eqref{laphi}, the estimate \eqref{z-in-Lp} provides that
 $\mathbb P\left( \Phi(z_m,N)<\infty\right)
\equiv \mathbb P\{C_6 \|z_m\|^8_{L^8(0,N;L^4)}<\infty\}=1$
and by \eqref{Cheb-z-in-Lp} we have that for any
$\eps,N>0$ there exists $M=M(\eps,N)>0$ such that
\[
\mathbb P\left( \Phi(z_m,N)>M\right) < \frac\eps 2 \frac 1{2^N}.
\]

Now we define the subset $S_{z,T,b}=\{\|z\|^8_{L^8(0,T;L^4)}\le b,
\|x\|^2_H\le b\}\subset \Omega$.
So for any $\eps,N>0$ there exists $b=b^{(\eps)}_N$ such that
for any $m$
\begin{equation}\label{piccoloS}
\mathbb P\{S_{z_m,N,b}\}>1-\frac \eps 2 \frac 1{2^N} .
\end{equation}
Notice that, on the set
$S_{z_m,N,b}$, the process
$\Phi(z_m,N)$ given in  \eqref{laphi} is bounded by
the constant $C_6 b$  and
the process $\Psi(z_m,N)$ given in \eqref{lapsi} is bounded by
the constant  $b+C_5 ( \frac N4+ b+N \|f\|^2_{H^{-1}})$, by using
$\|z\|^4_{L^4(0,N;L^4)}\le \frac N4+\|z\|^8_{L^8(0,N;L^4)}$.
Hence, on the set $S_{z_m,N,b}$ the process
$1+ \Psi(z_m,N)^2 (1+\Phi(z_m,N)^2) e ^{2\Phi(z_m,N)}$
appearing in the estimate \eqref{stime-vm-per-tight} for
$\n{ v_m}_N$ is bounded, $\mathbb P$-a.s.,  by
the constant
\[
C(b,N)
:=
1+[b+C_5 ( \tfrac N4+ b+N \|f\|^2_{H^{-1}})]^2 (1+(C_6b)^2) e ^{2C_6b}.
\]

This allows to find each element  $\alpha^{(\eps)}_N$ of the sequence
$\alpha^{(\eps)}$ satisfying
 \eqref{succ} as follows.
For short, we denote by $\eta(z_m,T)$ the RHS of \eqref{stime-vm-per-tight}.\\
By  \eqref{piccoloS},  we have
\begin{eqnarray*}
\mathbb P\left( \n{ v_m}_N>\alpha_N^{(\eps)}\right)
&\le
\mathbb P\left( \eta(z_m,N)>(\alpha_N^{(\eps)})^2\right)
\le \mathbb P\left( \Omega\setminus S_{z_m,N,b^{(\eps)}_N}\right)
\\&\qquad+\mathbb
P\left( S_{z_m,N,b^{(\eps)}_N}\cap
\{\eta(z_m,N)>(\alpha_N^{(\eps)})^2\}\right)
\\
&\le
\dfrac \eps 2 \dfrac 1{2^N}
+ \mathbb P\left(1_{S_{z_m,N,b^{(\eps)}_N}}
\eta(z_m,N)>(\alpha_N^{(\eps)})^2\right).
\end{eqnarray*}
So, we are left to find $\alpha^{(\eps)}_N$
such that the latter probability is bounded by   $\frac \eps 2 \frac 1{2^N}$.

By Chebyshev inequality we get
\[\begin{split}
 \mathbb P&\left( 1_{S_{z_m,N,b^{(\eps)}_N}} \eta(z_m,N)>(\alpha_N^{(\eps)})^2\right)\\
&\le
\tfrac 1{(\alpha_N^{(\eps)})^2}\mathbb E\left[1_{S_{z_m,N,b^{(\eps)}_N}}
  \eta(z_m,N)\right]\\
&\le \tfrac 1{(\alpha_N^{(\eps)})^2}
 \Big( (1+N^2)C(b^{(\eps)}_N,N)+1+\mathbb
 E\|z_m\|^4_{L^{\frac 83}(0,N;L^4)}
\\&\qquad\qquad
          +(1+N)\mathbb E\|z_m\|^2_{C^\beta([0,N];H^\delta)} +N^{\frac 32} \|f\|^2_{H^{-1}}
 \Big)
 \\
&\le \tfrac 1{(\alpha_N^{(\eps)})^2}
\Big(
  (1+N^2) C(b^{(\eps)}_N,N)
        +1+C_8^{\frac 32} (1+N)^{\frac 32}
\\&\qquad\qquad+(1+N)
         C_9^2 (1+N^{\frac{1-g}2-\beta-\frac \delta 2})^2 +N^{\frac 32} \|f\|^2_{H^{-1}}\Big)
\end{split}
\]
where we used  \eqref{z-in-Lp}-\eqref{z-in-C} to
 estimate the mean values involving the processes $z_m$. And finally we ask the
 latter quantity to be smaller than $ \frac \eps 2 \frac
1{2^N}$ to determine $\alpha_N^{(\eps)}$ so that \eqref{succ} holds true.
This proves the tightness.
\end{proof}

Now, we  prove
%THEOREM
\begin{theorem}\label{th:mis-inv-3d}
Let $d=3$.
If  $f \in H^{-1}$ and {\bf (G1)}-{\bf (G2)}-{\bf (G3)}-{\bf (G4)}
are satisfied,
then there exists at least one stationary martingale solution for \eqref{sns}.
\end{theorem}
\begin{proof}
For each $m\in \mathbb N$,  we fix an invariant measure $\mu_m$ for equation
\eqref{snsm} as given by Theorem \ref{th:mis-inv-3d-m} and
denote by $v_m$   the stationary
solution of \eqref{snsm} whose
marginal at any fixed time is $\mu_m$.
According to Proposition \ref{pro-tight}, the
sequence of the laws of the stationary processes $v_m$ is tight in the space
\[
Z=C(\mathbb R_+;H_{\mathrm{w}})\cap L^2_{\mathrm{loc}}(\mathbb R_+;H_{\mathrm{loc}})
\cap
L^{\frac 83}_{\mathrm{loc},\mathrm{w}}(\mathbb R_+;L^4) 
\cap C(\mathbb R_+;U^\prime).
\]

As in Section \ref{sect-SwF}, we use  the Jakubowski's generalization  of the
Skorokhod Theorem to  nonmetric spaces, see \cite{Ja1997}. Hence, there exist  a subsequence
$\{v_{m_j}\}_{j=1}^\infty$,
a stochastic basis $(\tilde \Omega, \tilde{\mathbb F},
\{\tilde{\mathbb F}_t\}_{t\ge 0}, \tilde{\mathbb P})$, $Z$-valued
Borel measurable variables $\tilde v$ and $\{\tilde v_j\}_{j=1}^\infty$ such that
for any $j\in \mathbb N$ the laws of $v_{m_j}$ and $\tilde v_j$ are the same
and $\tilde v_j$ converges to $\tilde v$ $\tilde {\mathbb P}$-a.s.
with the topology $\mathcal T$. Moreover,
this limit process $\tilde v$ is stationary in $H$. Indeed,
by the a.s. convergence in $C(\mathbb R_+;U^\prime)$
the process $\tilde v$ is stationary in $U^\prime$.
But, $\tilde v \in C(\mathbb R_+;H_{\mathrm{w}})$ so
each   $\tilde  v(t)$  is an $H$-valued random variable and therefore $\tilde v$
  is stationary in $H$.

Since each $\tilde v_j$ has the same law as $v_{m_j}$,
it is a martingale solution to the smoothed equation \eqref{snsm}; therefore
 for any $\phi \in H^a$ with $a>\frac 52$ (so that
$\|\nabla \phi\|_{L^\infty}\le C \|\phi\|_{H^a}$), we consider
the process
\begin{multline*}
M^\phi_j(t)=(\tilde v_j(t),\phi)_H -(\tilde v_j(0),\phi)_H
+\int_0^t (\tilde v_j(s),A\phi)_H ds \\
+\gamma \int_0^t   (\tilde v_j(s),\phi)_H ds
+\int_0^t \langle B_{m_j}(\tilde v_j(s),\tilde v_j(s)),\phi\rangle  ds
- \langle f,\phi\rangle t
\end{multline*}
Since  $M^\phi_j(t)$ has the same law as
$\langle \int_{0}^t G(v_{m_j}(s))\,dw(s),\phi\rangle$,
it is a martingale; in particular
\begin{equation}\label{mart1}
\tilde{\mathbb E}\big[(M^\phi_j(t)-M^\phi_j(s))
 h(\tilde v_j|_{[0,s]})\big]=0
\end{equation}
\begin{multline}\label{mart2}
\tilde{\mathbb E}\big[(M^\phi_j(t)M^\psi_j(t)-M^\phi_j(s)M^\psi_j(s))
 h(\tilde v_j|_{[0,s]})\big]
\\=
\tilde {\mathbb E}\big[(\int_s^t
\left(G(\tilde v_{j}(r))^*J^{2g}\phi, G(\tilde v_{j}(r))^*J^{2g}\psi\right)_{Y}dr)
 h(\tilde v_{j}|_{[0,s]})
\big]
\end{multline}
for any $0<s<t$, any $\phi,\psi \in H^a$ and any
bounded and continuous function $h:C([0,s];U^\prime)\to\mathbb R$.

When $j \to \infty$, for any $0<t$  we have
$M^\phi_j(t)\to M^\phi(t)$ $\tilde{\mathbb P}$-a.s.,  where
\begin{multline*}
M^\phi(t)=(\tilde v(t),\phi)_H -(\tilde v(0),\phi)_H
+\int_0^t (\tilde v(s),A\phi)_H ds \\
+\gamma \int_0^t   (\tilde v(s),\phi)_H ds
+\int_0^t \langle B(\tilde v(s),\tilde v(s)),\phi\rangle  ds
- \langle f,\phi\rangle t.
\end{multline*}
Indeed, the convergence of each term
is done as in \cite{noi} except for the trilinear term.
For this, we take $\phi$ regular enough and with compact support
(hence $\nabla \phi$ is bounded), let
us say that the support is contained in a centered ball of radius $R$;  we have

\begin{eqnarray*}
&&\hspace{-1truecm}\lefteqn{|\langle B_{m_j}(\tilde v_j(s),\tilde v_j(s)), \phi \rangle
- \langle B(\tilde v(s),\tilde v(s)), \phi \rangle|}
\\&\le& |\langle B_{m_j}(\tilde v_j(s)-\tilde v(s), \phi ),\tilde v_j(s)\rangle|\\
&+&|\langle B_{m_j}(\tilde v(s),\phi), \tilde v_j(s)- \tilde v(s)\rangle|
+|\langle B(\rho_{m_j}\ast \tilde v(s)-\tilde v(s), \phi),  \tilde v(s)\rangle|
\\&
\le&
C \|\tilde v_j(s)-\tilde v(s)\|_{H_R} \|\tilde v_j(s)\|_{H_R}
\|\nabla\phi\|_{L^\infty}\\&  +&
C \|\tilde v(s)\|_{H_R} \|\tilde v_j(s)-\tilde v(s)\|_{H_R}
\|\nabla\phi\|_{L^\infty}
\\&
 + & C
\|\rho_{m_j}\ast \tilde v(s)-\tilde v(s)\|_{H_R}\|\tilde v(s)\|_{H_R}
\|\nabla\phi\|_{L^\infty}.
\end{eqnarray*}
Using the convergence in the space $L^2(0,T;H_{\mathrm{loc}})$,
we pass into the limit as $m_j\to \infty$. % getting the convergence for a.e. $s$
%and then the convergence also of the time integral by dominated convergence.
Since each $\phi \in H^a$ can be approximated by a smooth and compactly supported
vector field, we
obtain the same limit for any $\phi \in H^a$, see Lemma B.1 in
\cite{BM2013}.

Therefore, to get the convergence of the LHS of \eqref{mart1} and
\eqref{mart2} we appeal to Vitali theorem, using the pathwise
convergence
and the uniform estimate
\[\begin{split}
\tilde {\mathbb E}\Big[|M^\phi_j(t)|^4\Big]
&=
{\mathbb E}\Big[|\langle \int_0^t G( v_{m_j}(r))dw(r),\phi\rangle|^4\Big]
\\
&\le
{\mathbb E} \Big\|\int_0^t G(v_{m_j}(r))dw(r)\Big\|_{H^{-g}}^4 \|\phi\|_{H^g}^4
\\
& \le
C  {\mathbb E} \Big(\int_0^t \|G( v_{m_j}(r))\|_{\gamma(Y;H^{-g})}^2dr\Big)^2
\|\phi\|_{H^g}^4
\\
& \le
C (K_{g,2} ^2t)^2 \|\phi\|_{H^g}^4
\end{split}
\]
thanks to {\bf (G1)}.

As far as the convergence of the RHS of \eqref{mart2} is concerned,
we recall from \cite{noi} the pathwise convergence of the cross
variance process
\[
\int_0^t \big( G(\tilde v_{j}(r))^*f_1, G(\tilde v_{j}(r))^*f_2\big)_Y dr
\to \int_0^t \big( G(\tilde v(r))^*f_1, G(\tilde v(r))^* f_2\big)_Y dr
\]
for any $f_1,f_2 \in H^{-g}$,
thanks to assumption {\bf (G3)}.
Again by means of Vitali theorem we get the convergence for the mean
value in the RHS of \eqref{mart2} since
$\|G(\tilde v_{j}(r))^*\|_{\gamma(H^{-g};Y)}=\|G(\tilde
v_{j}(r))\|_{\gamma(Y;H^{-g})}\le K_{g,2}$.

Therefore relationships \eqref{mart1} and \eqref{mart2} hold also for
the limit process $M^\phi$; they show that this is a martingale.
Therefore, with usual martingale representation theorem, see e.g. \cite{dpz},
we  conclude that
there exists a $Y$-cylindrical Wiener process $\tilde w$ such that
\[
\langle  M^\phi(t),\phi\rangle=\langle  \int_0^t G(\tilde v(s))\,d\tilde w(s),\phi\rangle.
\]
This concludes the proof that the limit
process $\tilde v$ is a stationary martingale solution of equation \eqref{sns}.
\end{proof}

\end{document}